\documentclass[a4paper,10pt]{amsart}

\usepackage{amsmath}
\usepackage{amsthm}
\usepackage{amsfonts}
\usepackage{amssymb}
\usepackage[all]{xy}
\usepackage{hyperref}
\usepackage{enumitem}
\usepackage{comment}
\usepackage{todonotes}
\usepackage{graphicx}
\usepackage[style=numeric,giveninits=true, sorting=nyt, backend=bibtex,doi=false,isbn=false,url=false,maxbibnames=50]{biblatex}
\newcommand{\ZZ}{\mathbb{Z}}

\newcommand{\QQ}{\mathcal{Q}}
\newcommand{\RR}{\mathbb{R}}

\newcommand{\ee}{\varepsilon}
\newcommand{\im}{\mathrm{Im}}
\newcommand{\Ker}{\mathrm{Ker}}
\newcommand{\n}{\mathfrak{n}}
\newcommand{\h}{\mathfrak{h}}

\newcommand{\degg}{\leq_{\deg}}
\DeclareMathOperator{\GL}{GL}

\DeclareMathOperator{\Hom}{Hom}

\newcommand{\bi}{\mathbf{i}}

\newcommand{\bd}{\mathbf{d}}

\newcommand{\be}{\mathbf{e}}

\newcommand{\g}{\mathfrak{g}}

\newcommand{\ve}{\varepsilon}

\newcommand{\wt}{\widetilde}

\newcommand{\form}{\langle-,-\rangle}

\usepackage{amssymb,amsmath,amsfonts,amscd,euscript}

\numberwithin{equation}{section}
\newtheorem{thm}{Theorem}[section]
\newtheorem*{thm*}{Theorem}
\newtheorem{prop}[thm]{Proposition}
\newtheorem{lem}[thm]{Lemma}
\newtheorem{cor}[thm]{Corollary}
\theoremstyle{remark}
\newtheorem{rem}[thm]{Remark}
\newtheorem{definition}[thm]{Definition}

\begin{document}
\title[Linear degenerations of symplectic flags]{Linear degenerate symplectic flag varieties: symmetric degenerations and PBW locus}
\author{Magdalena Boos, Giovanni Cerulli Irelli,  Xin Fang, Ghislain Fourier}
\address{Magdalena Boos:\newline Ruhr University Bochum, Faculty of Mathematics,  44780 Bochum, (Germany).} \email{magdalena.boos@gmx.net}
\address{Giovanni Cerulli Irelli:\newline
Department S.B.A.I.. Sapienza-Universit\`a di Roma. Via Scarpa 10, 00164, Roma (Italy)}
\email{giovanni.cerulliirelli@uniroma1.it}
\address{Xin Fang:\newline
Chair of Algebra and Representation Theory, RWTH Aachen University, Pontdriesch 10--16, 52062 Aachen}
\email{xinfang.math@gmail.com}
\address{Ghislain Fourier:\newline
Chair of Algebra and Representation Theory, RWTH Aachen University, Pontdriesch 10--16, 52062 Aachen}
\email{fourier@art.rwth-aachen.de}

\begin{abstract}
We conceptualize in the paper the linear degenerate symplectic flag varieties as symmetric degenerations within the framework of type $A$ equioriented quivers. First, in the larger context of symmetric degenerations, we give a self-contained proof of the equivalence of different degeneration orders. Furthermore, we investigate the PBW locus: geometric properties of the degenerate varieties in this locus are proved by realizing them from different perspectives.
\end{abstract}

\maketitle

\section{Introduction}
For a quiver $\QQ$, a representation $M$ of $\QQ$ and a fixed dimension vector $\mathbf{e}$, the quiver Grassmannian is the moduli space of the $\mathbf{e}$-dimensional subrepresentations of $M$. 
This represents a vast generalization of the notion of Grassmann varieties, which parametrize the $k$-dimensional subspaces in an $n$-dimensional space.
Quiver Grassmannians, with their reduced structures, are projective varieties\footnote{Projective varieties are not necessarily irreducible.} and furthermore, any projective variety is isomorphic to some quiver Grassmannian \cite{Rei13}. Therefore, few geometric properties are shared by all quiver Grassmannians, and one should specify a particular quiver or a certain family of quivers at the outset.
A specific example of a quiver Grassmannian is the flag variety, which can be described using subrepresentations for the type $A$ equioriented quiver.

Quiver Grassmannians for quivers of Dynkin type have gained a lot of attention \cite{CIEFR21, CIEM23,  CIEFF22, CFFFR20,  CIFR12, CIFR17, CIL, FR21, Fei12,   LS19, Mak19}.
For example in \cite{CFFFR}, quiver Grassmannians of the equioriented type $A$ quiver with arbitrary representation of fixed dimension vector are studied under the name \emph{linear degenerate flag varieties}.
They can be viewed as degenerations (as projective varieties) of the flag varieties: they are fibers of a proper family $\pi: X\rightarrow Y$, called the universal degenerate flag variety, whose generic fiber is the flag variety and the parameterizing space $Y$ is the representation variety of an equioriented type $A_n$ quiver. In \emph{loc.cit.}, the authors analyzed the locus where the fiber has minimal dimension (called flat locus); within it, the locus where the fiber is irreducible (called the flat irreducible locus); and within it, the locus where the fiber is naturally a Schubert variety (called the \emph{PBW locus}). The family $\pi$ is acted upon by a group $G$, and every $G$-orbit $G\cdot M$ on $Y$ is uniquely determined by a rank sequence $r^M:=(r_{i,j}^M|1\leq i\leq j\leq n)$. These loci are described in terms of rank sequences.\\

This current project is the first step in analyzing linear degenerate symplectic flag varieties. 
The symplectic flag varieties appear naturally as highest weight orbits of the action of the symplectic group, similarly to how the classical flag varieties appear as orbits of the general linear group. 

On the other hand, the symplectic flag variety is also naturally the fixed point set of an involution acting on the complete flag variety. 
In order to define a universal \emph{symplectic} degenerate flag variety, we consider an equioriented type $A$ quiver $\stackrel{\rightarrow}{A}_{\mathrm{odd}}$, having an odd number of vertices. 
In this case, the universal degenerate flag variety $\pi:X\rightarrow Y$ is acted upon by an involution ${}^\star$ whose fixed point family, denoted by $\pi^\star:X^\star\to Y^\star$, is the required universal symplectic degenerate flag variety. 
This variety is equivariant with respect to a closed subgroup $H$ of $G$. 
The points of $Y^\star$ are called symplectic representations of $\stackrel{\rightarrow}{A}$.

Our first result describes the $H$-orbits and their closures inside $Y^\star$ in terms of rank sequences. 
In fact, this construction is part of a general theory of quivers with self-duality, wherein this situation is termed the $(\stackrel{\rightarrow}{A}_{\mathrm{odd}},-1)$ type. 
We describe the orbit closures also in the other type $(\stackrel{\rightarrow}{A}_{\mathrm{even}},+1)$. 
In general, for a positive integer $n$ and a sign $\ee=\pm 1$ there is the notion of $\ee$-representation of $\stackrel{\rightarrow}{A}_n$, if $\ee=-1$ these are called symplectic and if $\ee=1$ they are called orthogonal (see Section~\ref{Sec:EE-rep}). 
In order to formulate the result, recall that the closures of the $G$-orbits in $Y$ are also described by ``cuts'' and ``shifts'' \cite{AbeasisDelFraEquioriented}. 
In order to give a combinatorial way to describe the closures of the $H$-orbits we introduce the notion of ``symmetric cuts''  and ``symmetric shifts'' (see Definition~\ref{Def:SymmCutsShifts}). The following is our first main result (see Theorem~\ref{Thm:MainEquioriented} and Theorem~\ref{Thm:ee-degenerationRanks} for precise statements):


\begin{thm*}
Let $n\geq 1$ be a positive integer,  $\stackrel{\rightarrow}{A}_n$ an equioriented quiver whose underlying graph is the Dynkin diagram $A_n$, and let $\sigma$ denote its  non-trivial involution. 
Let $M$ be a representation of $\stackrel{\rightarrow}{A}_n$. 
\begin{enumerate}
    \item The intersection $G\cdot M\cap Y^\star$ is non-empty if and only if the rank sequence $r^M=(r_{i,j}^M)$ of $M$ satisfies the following relations: 
\begin{eqnarray}
&&r_{i,j}^M=r_{\sigma(j),\sigma(i)}^M;\\
&&\text{in type }(\stackrel{\rightarrow}{A}_{\mathrm{odd}},-1)\text{ and in type }(\stackrel{\rightarrow}{A}_{\mathrm{even}},+1),\text{ the ranks }  r_{i,\sigma(i)}^M\text{ are even}.
\end{eqnarray}
\item In the $(\stackrel{\rightarrow}{A}_{\mathrm{odd}},-1)$ and $(\stackrel{\rightarrow}{A}_{\mathrm{even}},+1)$ types,
an $\ee$-representation $N$ lies in the closure of the $H$-orbit of an $\ee$-representation $M$ if and only if  $r^M\geq r^N$ if and only if $N$ is obtained from $M$ by a finite sequence of symmetric cuts and symmetric shifts.
\end{enumerate} 
\end{thm*}
We provide moreover an algorithm to compute for a given symmetric degeneration $M\degg^\ee N$ a sequence of one-parameter subgroups from $M$ to $N$ (see Section~\ref{Sec_Algorithm} for some examples).
\bigskip

Let us now turn our attention back to the universal symplectic degenerate flag variety. We consider some special fibers that we call \textit{PBW}-degenerations of the symplectic flag variety, in analogy to \cite{CFFFR}. They are parametrized by a subset $\bi$ of $[n]$, where the quiver has $2n-1$ vertices. In \cite{CFFFR}, such a subset $\bi$ is naturally associated with an element $w_\bi$ of the symmetric group. Similarly, we associate it here with an element in the Weyl group of the  symplectic group.
The second main result of the paper is the following
\begin{thm*}
    A \textit{PBW}-degeneration of the symplectic flag variety is irreducible, reduced, normal, Cohen-Macaulay, Frobenius splitting and has rational singularities.
\end{thm*}
We deduce this from the following theorem (Theorem~\ref{Thm:Iso}) whose proof is given in Section~\ref{Sec:Proof}.
\begin{thm*}
The following varieties are isomorphic:
\begin{enumerate}
\item the symmetric degeneration $\mathrm{Sp}\mathcal{F}_{2n}^\bi$;
\item the Lagrangian quiver Grassmannian $\mathrm{Lag}_{\be}(M^\bi\oplus\nabla (M^\bi))$;
\item the Schubert variety $X_{w_\bi}$;
\item for any $\bd\in \mathrm{relint}(F^\bi)$, the $\bd$-degenerate flag variety $\mathcal{F}^\bd(\rho)$, where $\rho=\varpi_1+\ldots+\varpi_n$ is the sum of all fundamental weights.
\end{enumerate}
\end{thm*}

The study of other fibers of the universal symplectic degenerate flag variety is part of a future project. Here the similarities to the non-symplectic case are less obvious and the techniques to study Lagrangian quiver Grassmannians have to be developed further. 

The paper is structured as follows: in Section~\ref{Sec:Equioriented} we recall results on symmetric representation theory, properties of rank sequences and provide the proof of the first Theorem above.  In Section~\ref{Sec:PBW-locus} we introduce the Langragian quiver Grassmannians, the combinatorics of the Schubert varieties, recall the abelianizations of symplectic flag varieties and state our second main result, which will be then proved in Section~\ref{Sec:Proof}. We conclude the paper with the algorithm that computes the degeneration sequences and some examples in Section~\ref{Sec_Algorithm}.

\vspace{10pt}
\noindent \textbf{Acknowledgements:} We thank Peter Littelmann for all the inspirations he gave us during the years. 

The work of MB was sponsored by DFG Forschungsstipendium BO 5359/1-1, DFG R\"uckkehrstipendium BO 5359/3-1 and DFG Sachbeihilfe BO 5359/2-1. The work of GCI was sponsored by PRIN2022-SQUARE and by INDAM-GNSAGA. The work of GF is funded by the Deutsche
Forschungsgemeinschaft (DFG, German Research Foundation): “Symbolic Tools in
Mathematics and their Application” (TRR 195, project-ID 286237555).

\section{The equioriented type A quivers}\label{Sec:Equioriented}

In the paper we fix $\mathbb{C}$, the field of complex numbers, as the base field. 

\subsection{Recollection: equioriented type \texorpdfstring{$A$}{A} quiver}

Let $n\geq 2$ be a positive integer. We denote by  $\stackrel{\rightarrow}{A_n}$ the equioriented Dynkin quiver of type $A_n$ with the following labeling of vertices and arrows: 
\begin{equation}\label{Eq:QuiverEqui}
\xymatrix{
\stackrel{\rightarrow}{A_n}:1\ar^{\alpha_1}[r]&2\ar^{\alpha_2}[r]&\cdots\ar[r]&n-1\ar^{\alpha_{n-1}}[r]&n.
}
\end{equation}
We denote by $\sigma$ the anti-involution defined on vertices by $\sigma(i)=n+1-i$ and on arrows by $\sigma(\alpha_i)=\alpha_{n-1}$. The pair $(\stackrel{\rightarrow}{A_n},\sigma)$ is the typical example of a \emph{quiver with self-duality} or \emph{symmetric quiver} (in the terminology of Derksen and Weyman \cite{DW}). In this section, we recall some well known facts about (symmetric) representation theory of $(\stackrel{\rightarrow}{A_n},\sigma)$ (see \cite{BCI}). 

The category of $\stackrel{\rightarrow}{A_n}$-representations is endowed with a self-duality that we denote by $\nabla$, defined naturally as follows: for a given $\stackrel{\rightarrow}{A_n}$-representation 
\begin{equation}\label{Eq:MEqui}
\xymatrix{
M:&M_1\ar^{f_1}[r]&M_2\ar^{f_2}[r]&\cdots \ar^{f_{n-2}}[r]&M_{n-1}\ar^{f_{n-1}}[r]&M_n,
}
\end{equation}
its (graded) dual is 
\begin{equation}\label{Eq:NablaMEqui}
\xymatrix{
\nabla M:&M_n^\ast\ar^{-f_{n-1}^\ast}[r]&M_{n-1}^\ast\ar^{-f_{n-2}^\ast}[r]&\cdots\ar^{-f_{2}^\ast}[r]&M_{2}^\ast\ar^{-f_{1}^\ast}[r]&M_1^\ast
}
\end{equation}
where we denote by $V^\ast= \Hom(V,\mathbb{C})$ the linear dual of a vector space $V$ and by $f^\ast$ the linear dual of a linear map $f$. This explains why $(\stackrel{\rightarrow}{A_n},\sigma)$ is called a quiver with self-duality. We denote by $\mathbf{dim}\,M=(\mathrm{dim}\,M_1,\cdots, \mathrm{dim}\, M_n)$ the dimension vector of $M$. It is well-known and easy to prove that the indecomposable representations of $\stackrel{\rightarrow}{A_n}$ are in bijection with the set $\{\ee_{i,j}|\, 1\leq i\leq j\leq n\}$ of positive roots of a simple Lie algebra of type $A_n$. For every ordered pair of vertices $1\leq i\leq j\leq n$, we denote by $U_{i,j}$ the indecomposable representation whose dimension vector is $\varepsilon_{i,j}=\ee_i+\cdots+ \ee_j$, where $\ee_i$ denotes the $i$-th element of the standard basis of $\RR^n$. We sometimes think of $U_{i,j}$ as the segment $[i,j]$. For every vertex $k$, the simple at $k$ is $S_k:=U_{k,k}$, its projective cover is $P_k:=U_{k,n}$ and its injective envelope is $I_k:=U_{1,k}$. We notice that  $\nabla U_{i,j}\simeq U_{\sigma(j),\sigma(i)}$, in particular $\nabla P_k\simeq I_{\sigma(k)}$.

Every $\stackrel{\rightarrow}{A_n}$-representation $M$ can be written in a unique way (up to reordering of summands) as a direct sum $M=\bigoplus_{1\leq i\leq j\leq n} U_{i,j}^{m_{i,j}}$ where $m_{i,j}$ denotes the multiplicity of $U_{i,j}$ as a direct summand of $M$. Thus $M$ can be thought of a collections of segments which form its \emph{coefficient quiver}.  We denote by $\mu(M)=(m_{i,j}|\, 1\leq i\leq j\leq n)$ the \emph{multiplicity sequence} of $M$.
 
For every $1\leq i<j\leq n$, we denote by $f_{i,j}^M=f_{j-1}\circ\cdots\circ f_i:M_i\rightarrow M_j$ the composite map and by $f_{i,i}^M=\mathrm{Id}_{M_i}$ the identity on $M_i$.  The \emph{rank sequence} of $M$ is the collection $r^M=(r_{i,j}^M)_{1\leq i\leq j\leq n}$ of the ranks $r_{i,j}^M=\mathrm{rank}(f_{i,j})$.  For example, the rank sequence of $U_{i,j}$ is 
$$
r^{U_{i,j}}_{k,\ell}=\left\{
\begin{array}{cc}
1&\textrm{if }i\leq k\leq \ell\leq j,\\
0&\textrm{otherwise}.
\end{array}
\right.
$$
We notice that 
\begin{equation}\label{Eq:RankNable}
r_{k,\ell}^{\nabla M}=r_{\sigma(\ell),\sigma(k)}^M, \textrm{ for every }1\leq k\leq\ell\leq n.
\end{equation} 
Throughout the paper we use the convention that $r_{i,j}^M=0$ for $i=0$ or for $j=n+1$ and $r_{i,j}^M=+\infty$ if $i>j$.

We write the rank sequences as  upper-triangular matrices, without writing down the lower part. We include the components $r_{i,i}:=\dim M_i$ on the diagonal. For example, the following is the rank sequence of the representation of $\stackrel{\rightarrow}{A_5}$ whose coefficient quiver is shown on the right (all the edges of the coefficient quiver point from left to right): 
$$
\begin{array}{cc}
\xymatrix@R=0pt@C=0pt{
1&1&1&1&0\\
  &2&2&2&1\\
  &  &4&2&1\\
  &  &  &2&1\\
  &  &  &  &1}
&
\xymatrix@R=0pt@C=5pt{
\bullet\ar@{-}[r]&\bullet\ar@{-}[r]&\bullet\ar@{-}[r]&\bullet&\\
                        &\bullet\ar@{-}[r]&\bullet\ar@{-}[r]&\bullet\ar@{-}[r]&\bullet\\
                        &                         &\bullet&&\\
                        &                         &\bullet&&}
\end{array}
$$

\begin{lem}\label{Lem:RankMultiplicity}
A sequence of non-negative integer $r=(r_{i,j}|\, 1\leq i\leq j\leq n)$ is the rank sequence of a representation of $\stackrel{\rightarrow}{A_n}$ if and only if it satisfies the inequalities 
\begin{eqnarray}\label{Eq:RankSeq1}
r_{i,j}\geq r_{i,j+1},&&\forall1\leq i\leq j<n;\\\label{Eq:RankSeq2}
r_{i-1,j}\leq r_{i,j},&&\forall1< i\leq j\leq n;\\\label{Eq:RankSeq3}
r_{i-1,j}-r_{i-1,j+1}\leq r_{i,j}-r_{i,j+1}, &&\forall1< i< j< n.
\end{eqnarray}
Moreover, the relation between the multiplicity sequence $\mu(M)=(m_{i,j})$ and the rank sequence $r^M$ of an $\stackrel{\rightarrow}{A_n}$-representation $M$ is 
\begin{equation}\label{Eq:MultiplicityRank}
m_{i,j}=r_{i,j}^M-r_{i,j+1}^M-r_{i-1,j}^M+r_{i-1,j+1}^M
\end{equation}
with the convention that $r_{-,n+1}^M=0=r_{0,-}^M$. In particular, two $\stackrel{\rightarrow}{A_n}$-representations are isomorphic if and only if they have the same rank sequence. 
\end{lem}
\begin{proof}
If $M=((M_i)_{i=1}^n, (f_i:M_i\rightarrow M_{i+1})_{i=1}^{n-1})$ is a representation of $\stackrel{\rightarrow}{A_n}$ then
$$
\begin{array}{lr}
f_j(\mathrm{Im}(f_{i,j}))\subseteq \im(f_{i,j+1})&(1\leq i\leq j<n),\\
\im(f_{i-1,j})\subseteq \im(f_{i,j})&(1< i\leq j\leq n),\\
\im(f_{i-1,j})\cap\ker(f_j)\subseteq \im(f_{i,j})\cap\ker(f_j)&(1< i\leq j<n).
\end{array}
$$ 
The inequalities \eqref{Eq:RankSeq1} and \eqref{Eq:RankSeq2} follow directly; the inequality \eqref{Eq:RankSeq3} follows from the fact that $r_{i,j}^M-r_{i,j+1}^M=\dim\im(f_{i,j})\cap\ker(f_j)$.  
Equation~\eqref{Eq:MultiplicityRank} is obtained by inverting the following obvious formula
\begin{equation}\label{Eq:RankMultiplicity}
r^M_{i,j}=\sum_{k\leq i, j\leq \ell}m_{k,\ell}.
\end{equation}
To prove the converse, we let $r=(r_{i,j})_{1\leq i\leq j\leq n}$ to be a sequence of non-negative integers satisfying the inequalities \eqref{Eq:RankSeq1}, \eqref{Eq:RankSeq2} and \eqref{Eq:RankSeq3}. In particular, for every $1\leq i\leq j\leq n$, the integer $m_{i,j}:=r_{i,j}-r_{i,j+1}-r_{i-1,j}+r_{i-1,j+1}$ is non-negative. Thus, by \eqref{Eq:RankMultiplicity}, the representation $M=\bigoplus U_{i,j}^{m_{i,j}}$ is such that  $r^M=r$.
\end{proof}
The Ringel form of $\stackrel{\rightarrow}{A_n}$ is the form $\langle-,-\rangle_n:\ZZ^n\times\ZZ^n\rightarrow \ZZ$ defined by $\langle\mathbf{d},\mathbf{e}\rangle_n=\sum_{i=1}^n\mathbf{d}_i\mathbf{e}_i-\sum_{i=1}^{n-1}\mathbf{d}_i\mathbf{e}_{i+1}$. Given two $\stackrel{\rightarrow}{A_n}$-representations $M$ and $N$ we denote by $[M,N]=\dim\Hom(M,N)$ and $[M,N]^1=\dim \mathrm{Ext}^1(M,N)$. It is well-known that 
\begin{equation}\label{Eq:EulerForm}
[M,N]-[M,N]^1=\langle\textbf{dim}\,M,\textbf{dim}\,N\rangle_n.
\end{equation}

We have
\begin{equation}\label{Eq:HomIndEquioriented}
[U_{i,j},U_{k,\ell}]=\left\{\begin{array}{cc}1&\textrm{if }k\leq i\leq \ell\leq j,\\0&\mathrm{otherwise.}\end{array}\right.
\end{equation}
The AR-functors act on the indecomposables by 
$\tau U_{i,j}=U_{i+1,j+1}$ (for $j<n$) and $\tau^-U_{k,\ell}=U_{k-1,\ell-1}$ (for $k>1$). Then from the AR-formulas we get the analogous formulas for the Ext-spaces:
\begin{equation}\label{Eq:ExtIndEquioriented}
[U_{k,\ell},U_{i,j}]^1=\left\{\begin{array}{cc}1&\textrm{if }k\leq i\leq \ell\leq j,\\0&\mathrm{otherwise.}\end{array}\right.
\end{equation}

\begin{lem}\label{Lem:Embeddings}
Let $M$ be a $\stackrel{\rightarrow}{A_n}$-representation and let $1\leq i\leq j\leq n$. Then 
\begin{enumerate}
\item $U_{i,j}$ embeds into $M$ if and only if $r_{i,j}^M-r_{i,j+1}^M>0$.
\item $U_{i,j}$ is a quotient of $M$ if and only if $r_{i,j}^M-r_{i-1,j}^M>0$.
\item $U_{i,j}$ is a direct summand of $M$ if and only if $r_{i,j}^M-r_{i,j+1}^M>r_{i-1,j}^M-r_{i-1,j+1}^M$. 
\end{enumerate}
(We use again the convention that $r_{-,n+1}^M=0=r_{0,-}^M$.)
\end{lem}
\begin{proof}
The indecomposable representation $L=U_{i,j}$ embeds into $M$ if and only if there exists a vector $v\in M_i$ such that $f_{i,j}(v)$ is non-zero and $f_j(f_{i,j}(v))=0$. Thus, $L$ embeds into $M$ if and only if  $\im(f_{i,j})\cap\ker(f_j)$ is non-zero, proving (i). 

To prove (ii), we observe that $L=U_{i,j}$ is a quotient of $M$ if and only if there exists a  vector $v\in M_i$ which does not lie in the image of $f_{i-1}$ and $f_{i,j}(v)\neq 0$. This is the case, if and only if $f_{i,j}(\im f_{i-1})\subsetneq \im (f_{i,j})$.

Part (iii) has already been proved in  Lemma~\ref{Lem:RankMultiplicity}.
\end{proof}
\begin{lem}
Let $M$ be a $\stackrel{\rightarrow}{A_n}$-representation. Then for every $1\leq k\leq \ell\leq n$, 
\begin{eqnarray}\label{Eq:EquiorientedHom(M,U)Rank}
[M,U_{k,\ell}]&=&r_{\ell,\ell}^M-r_{k-1,\ell}^M,\\\label{Eq:EquiorientedHom(U,M)Rank}
[U_{k,\ell},M]&=&r_{k,k}^M-r_{k,\ell+1}^M,
\end{eqnarray}
with the convention that $r_{0,-}^M=0=r_{-,n+1}^M$. 
\end{lem}
\begin{proof}
By \eqref{Eq:HomIndEquioriented} we have
\begin{eqnarray*}
[M,U_{k,\ell}]&=&\sum_{k\leq i\leq \ell\leq j}m_{i,j}\\
&=&\sum_{k\leq i\leq \ell\leq j}r_{i,j}^M-r_{i,j+1}^M-r_{i-1,j}^M+r_{i-1,j+1}^M\\
&=&r_{\ell,\ell}^M-r_{k-1,\ell}^M.
\end{eqnarray*}
The proof of \eqref{Eq:EquiorientedHom(U,M)Rank} is similar. 
\end{proof}
Given a graded vector space $M^0=M_1\oplus\cdots \oplus M_n$, we denote by $R(M^0)=\prod_{i=1}^{n-1}\Hom(M_i,M_{i+1})$ the affine space parametrizing representations of $\stackrel{\rightarrow}{A_n}$ with underlying vector space $M^0$. The group $\GL^\bullet(M^0)=\prod_{i=1}^n \GL(M_i)$ acts on $R(M^0)$ by change of basis and the $\GL^\bullet(M^0)$-orbits correspond to isoclasses of representations. Given $M, N\in R(M^0)$ we denote by  $\leq_{\mathrm{deg}}$ the \emph{degeneration order}, defined by $M\degg N$ if $N\in \overline{\GL^\bullet(M^0)\cdot M}$ and by $r^M\geq r^N$ the rank ordering defined by $r_{i,j}^M\geq r_{i,j}^N$ and $r_{i,i}^M=r_{i,i}^N$ for every $1\leq i\leq j\leq n$. There is a third ordering, defined by Abeasis and Del Fra \cite[Section~3]{AbeasisDelFraEquioriented} which is called the \emph{combinatorial ordering} defined  by $M\leq_c N$ if the coefficient quiver of $N$ is obtained from the coefficient quiver of $M$ by performing a finite number of \emph{cuts} and \emph{shifts} (see Table~\ref{Table:Cut} and \ref{Table:Shift}).
\begin{table}[htbp]
\begin{tabular}{|c|c|}
\hline
$
\xymatrix@R=0pt@C=7pt{
t&&q-1&q&&s\\
*{\bullet}\ar@{-}[rrrrr]&&*{\bullet}&*{\bullet}&&*{\bullet}
}
$
&
$
\xymatrix@R=0pt@C=7pt{
t&&q-1&q&&s\\
*{\bullet}\ar@{-}[rr]&&*{\bullet}&*{\bullet}\ar@{-}[rr]&&*{\bullet}
}
$
\\
$M$&$N$\\\hline
\end{tabular}
\caption{A cut from $M$ to $N$ at column $q-1$ of the segment $[t,s]\subset M$}
\label{Table:Cut}
\end{table}
\begin{table}[htbp]
\begin{tabular}{|c|c|}
\hline
$
\xymatrix@R=0pt@C=7pt{
t&q&&r&s\\
*{\bullet}\ar@{-}[rrrr]&\bullet&&\bullet&*{\bullet}\\
&*{\bullet}\ar@{-}[rr]&&*{\bullet}&
}
$
&
$
\xymatrix@R=0pt@C=7pt{
t&q&&r&s\\
*{\bullet}\ar@{-}[rrr]&\bullet&&*{\bullet}&\\
&*{\bullet}\ar@{-}[rrr]&&\bullet&*{\bullet}
}
$
\\
$M$&$N$\\\hline
\end{tabular}
\caption{A shift from $M$ to $N$ of the segments $[r,s]$ from $[t,s]$ to $[q,r]$}
\label{Table:Shift}
\end{table}
\begin{lem}[{\cite[Lemmas 3.3 and 3.4]{AbeasisDelFraEquioriented}}]\label{Lem:RankCutShift} Let $M,N\in R(M^0)$. 
\begin{enumerate}
\item $N$ is obtained from $M$ by a cut at column $q-1$ of the segment $[t,s]\subset M$ if and only if 
$$
r^N_{k,\ell}=
\left\{
\begin{array}{ll}
r^M_{k,\ell}-1&\textrm{if }t\leq k<q\leq \ell\leq s,\\&\\
r^M_{k,\ell}&\textrm{otherwise.}
\end{array}
\right.
$$ 
\item $N$ is obtained from $M$ by a shift of the segment $[r,s]$ from $[t,s]$ to $[q,r]$ if and only if 
$$
r^N_{k,\ell}=
\left\{
\begin{array}{ll}
r^M_{k,\ell}-1&\textrm{if }t\leq k< q\leq r< \ell\leq s\\&\\
r^M_{k,\ell}&\textrm{otherwise.}
\end{array}
\right.
$$ 
\end{enumerate}
\end{lem}
We recall from \cite{Bongartz} that for algebras of finite representation type (e.g. every Dynkin quiver), every subrepresentation $L$ of a representation  $M$ admits a generic quotient, i.e. 
A quotient $Q$ of $M$ by $L$ is called generic if its rank sequence is maximal among all possible quotients of $M$ by $L$. It is well-known that generic quotients exist \cite[Theorem~2.4]{Bongartz}. The next result describes the rank sequence of the generic quotient of a representation $M$ by an \emph{indecomposable} representation $L$. 

\begin{prop}\label{Prop:RankSequenceQuotient}
Let $M$ be a $\stackrel{\rightarrow}{A_n}$-representation and let $L=U_{q,s}$. Suppose that $L$ embeds into $M$ and let $\iota:L\hookrightarrow M$ be a generic embedding. Let $Q=M/\iota(L)$ be the generic quotient. Then  the rank sequence $r^Q=(r_{k,\ell}^Q)$ of $Q$ is obtained from the rank sequence $r^M=(r_{k,\ell}^M)$ of $M$ as follows:
\begin{equation}\label{Eq:RankGenericQuotient}
r_{k,\ell}^Q=\left\{
\begin{array}{ll}
r_{k,\ell}^M-1&\text{if }r_{q,\ell}^M-r_{q,s+1}^M\leq r_{k,\ell}^M-r_{k,s+1}^M\text{ and }q\leq\ell\leq s,\\
r_{k,\ell}^M&\text{otherwise.}
\end{array}
\right.
\end{equation}
\end{prop}
\begin{proof}
An embedding of $L$ into $M$ is the span of $\{v_q, v_{q+1},\cdots, v_{s}\}$ where $v_\ell\in M_\ell$ is  non-zero, $v_\ell=f_{q,\ell}(v_q)$,  $v_\ell\in \im(f_{q,\ell})\cap\ker(f_{\ell,s+1})$ and $v_\ell\not\in\ker(f_{\ell,s})$, for every $q\leq \ell\leq s$. 
We notice that in $M_\ell$ there is a chain of vector subspaces:
\begin{equation}\label{Eq:Chain}
(\im f_{1,\ell}\cap\ker f_{\ell,s+1})\subseteq (\im f_{2,\ell}\cap\ker f_{\ell,s+1})\subseteq\cdots \subseteq(\im f_{q,\ell}\cap\ker f_{\ell,s+1}).
\end{equation}
In order $\iota(L)$ to be generic, the quotient $Q=M/\iota(L)$ must have the greatest rank sequence among all possible quotients of $M$ by $L$. 
Thus a generic $v_\ell$ as above does not belong to a subspace $(\im f_{k,\ell}\cap\ker f_{\ell,s+1})$ which is strictly contained in $(\im f_{q,\ell}\cap\ker f_{\ell,s+1})$. Since 
$
\dim(\im f_{k,\ell}\cap\ker f_{\ell,s+1})=r_{k,\ell}^M-r_{k,s+1}^M
$
we see that $r_{k,\ell}^M-r_{k,s+1}^M\leq r_{q,\ell}^M-r_{q,s+1}^M$ for $k\leq q\leq \ell\leq s$. From this we get the first condition of \eqref{Eq:RankGenericQuotient} (for $k\leq q$). Suppose now that $q<k\leq \ell\leq s$. Then $(\im f_{q,\ell}\cap\ker f_{\ell,s+1})\subseteq \im f_{k,\ell}\cap\ker f_{\ell,s+1}$ and hence $v_\ell$ must belong to this subspace. We notice that $r_{q,\ell}^M-r_{q,s+1}^M\leq r_{k,\ell}^M-r_{k,s+1}^M$ (for $q<k\leq s$) and hence we get the first condition of \eqref{Eq:RankGenericQuotient}. If $\ell<q$ or $\ell>s$ there are no conditions to impose. This proves \eqref{Eq:RankGenericQuotient}. 
\end{proof}

\begin{cor}
Let $M$ be a $\stackrel{\rightarrow}{A_n}$-representation and let $L=U_{q,s}$. Suppose that $L$ embeds into $M$ and let $\iota:L\hookrightarrow M$ be a generic embedding. Let $Q=M/\iota(L)$ be the generic quotient. Suppose that $L$ is not a direct summand of $M$. Then $q>1$ and the rank sequence of $L\oplus Q$ is 
\begin{equation}\label{Eq:RankGenericDegeneration}
r_{k,\ell}^{L\oplus Q}=\left\{
\begin{array}{ll}
r_{k,\ell}^M-1&\text{if }r_{q,\ell}^M-r_{q,s+1}^M= r_{k,\ell}^M-r_{k,s+1}^M\text{ and }k<q\leq\ell\leq s,\\
r_{k,\ell}^M&\text{otherwise.}
\end{array}
\right.
\end{equation}
In particular, $M\leq_c L\oplus Q$. 
\end{cor}
\begin{proof}
Since $r^{L\oplus Q}=r^L+r^Q$, formula \eqref{Eq:RankGenericDegeneration} follows directly from \eqref{Eq:RankGenericQuotient}. Let us prove that $M\leq_c L\oplus Q$. If $L$ is a direct summand of $M$ then $M\simeq L\oplus Q$ and we are done. Let us assume that $L$ is not a direct summand of $M$. Then, $q>1$, because otherwise $L$ would be injective and hence $\iota$ would be a split mono-morphism. Let us consider the function 
$$
f(k,\ell)=r_{q,\ell}^M-r_{q,s+1}^M- r_{k,\ell}^M+r_{k,s+1}^M.
$$
By Lemma~\ref{Lem:Embeddings}, $f(q-1,s)=0$ because $L$ is not a direct summand of $M$. Moreover, again by Lemma~\ref{Lem:Embeddings}, the following formulas holds:
\begin{eqnarray}\label{eq:f(k,ell)1}
&&f(k,\ell)=f(k,\ell+1)+(r^M_{q,\ell}-r^M_{q,\ell+1}-r^M_{k,\ell}+r^M_{k,\ell+1})\geq f(k,\ell+1)\geq0;\\\label{eq:f(k,ell)2}
&&f(k,\ell)=f(k+1,\ell)+(r^M_{k+1,\ell}-r^M_{k+1,s+1}-r^M_{k,\ell}+r^M_{k,s+1})\geq f(k+1,\ell)\geq0.
\end{eqnarray}
It follows that if $f(k,\ell)=0$ then $f(k,\ell+1)=f(k+1,\ell)$. Let $q_1\in [q,s]$ be the minimal index such that there exists $k\in[1,q-1]$ such that $f(k,q_1)=0$ and let $t_1$ be the minimal index such that $f(t_1,q_1)=0$.  Similarly, let $t_2\in [1,q-1]$ be the minimal index such that there exists $\ell\in[q,s]$ such that $f(k,\ell)=0$ and let $q_2$ be the minimal index such that $f(t_2,q_2)=0$. We notice that by \eqref{eq:f(k,ell)1} and \eqref{eq:f(k,ell)2}, $q\leq q_1\leq q_2\leq s$ and $t_2\leq t_1<q$. We notice that in view of Lemma~\ref{Lem:Embeddings}, $U_{t_1,q_2-1}$ and $U_{t_2,s}$ are direct summands of $M$.
We can rewrite \eqref{Eq:RankGenericDegeneration} as
\begin{equation}\label{Eq:RankGenericDegenerationCutsShift}
r_{k,\ell}^{L\oplus Q}=\left\{
\begin{array}{ll}
r_{k,\ell}^M-1&\text{if }t_1\leq k<q\leq q_1\leq \ell< q_2,\\
r_{k,\ell}^M-1&\text{if }t_2\leq <q\leq q_2\leq \ell\leq s,\\
r_{k,\ell}^M&\text{otherwise}.
\end{array}
\right.
\end{equation}
Now it is easy to conclude from Lemma~\ref{Lem:RankCutShift}: if $q_1=q_2$ then $t_1=t_2$; in this case $L\oplus Q$ is obtained from $M$ by a cut (at column $q-1$ of the segment $[t_1,s]$) if $q=q_1=q_2$, and by a shift (of the segment $[q_1-1,s]$ from $[t_1,s]$  to $[q,q_1]$) if $q<q_1=q_2$. If $q_1<q_2$ then $L\oplus Q$ is obtained from $M$ by performing first a cut (at column $q-1$ of the segment $[t_1,q_2-1]$) and then a shift (of the segment $[q,q_2-1]$ from $[t_2,s]$ to $[q,q_2-1]$) if $q=q_1$, and by performing a shift (of the segment $[q_1-1,q_2-1]$ from $[t_1,q_2-1]$ to $[q,q_1-1]$) and then another shift (of the segment $[q_2-1,s]$ from $[t_2,s]$ to $[q,q_2-1]$) if $q<q_1$. 
\end{proof}

\subsection{\texorpdfstring{$\ee$}{epsilon}-representations}\label{Sec:EE-rep}

Let $\ee$ be either $1$ or $-1$. For an $\ee$-form on a $\mathbb{C}$-vector space $W$ we mean a bilinear form $\langle-,-\rangle:W\times W\to\mathbb{C}$ satisfying $\langle v,w\rangle=\ve\langle w,v\rangle$ for every $v,w\in W$. Recall from \cite{DW, BCI} that a $\sigma$-compatible $\ee$-form on $M^0$ is a non-degenerate $\ee$-form such that the restriction of the form on $M_i\times M_j$ is zero unless $j=\sigma(i)$. We notice that if $M^0$ admits a $\sigma$-compatible $\ee$-form, then $M_i$ and $M_{\sigma(i)}$ are dual to each other, in particular they have the same dimension. Moreover, if $n=2h-1$ is odd then $M_{h}$ inherits a non-degenerate $\ee$-form, thus it is even dimensional if $\ee=-1$.
Let us suppose that $M^0$ is endowed with a $\sigma$-compatible $\ee$-form $\form$.  We consider the involution ${}^\star:R(M^0)\rightarrow R(M^0)$ which associates to $f=(f_i:M_i\rightarrow M_{i+1})$ its anti-adjoint $f^\star=(f^\star_i)\in R(M^0)$ given by the formula
$$
\langle f_i(v),w\rangle+\langle v, f^\star_{\sigma(i)}(w)\rangle=0
$$ 
for every $i=1,\cdots, n-1$, $v\in M_i$ and $w\in M_{\sigma(i+1)}$. We notice that $(M^0,f^\star)\simeq \nabla(M^0,f)$.
The variety of ${}^\star$-fixed points is denoted 
by
$$
R(M^0;\form)=\{f\in R(M^0)|\, f=f^\star\};
$$
and point in $R(M^0;\form)$ are called \emph{$\ee$-representations} of $(\stackrel{\rightarrow}{A_n},\sigma)$ with underlying space $(M^0,\form)$. Sometimes, $1$-representations are called orthogonal and $(-1)$-representations are called symplectic. We use the notation $(\stackrel{\rightarrow}{A}_\bullet,\ee)$ to distinguish the four cases $(\stackrel{\rightarrow}{A}_{odd},\pm1)$ and $(\stackrel{\rightarrow}{A}_{even},\pm1)$.
The group of isometries
$$
\GL^\bullet(M^0,\form)=\{g\in \GL^\bullet(M^0)|\, \langle g v, g w\rangle=\langle v,w\rangle,\, \forall v,w\in M^0\}
$$
of $(M^0,\form)$ acts on $R(M^0;\form)$ by change of basis  and by \cite[Theorem~2.6]{DW} (see also \cite[Theorem~2.5]{BCI}), we have that, for every $M\in R(M^0;\form)$,
\begin{equation}\label{Eq:MWZTheorem}
\GL^\bullet(M^0,\form)\cdot M=\GL^\bullet(M^0)\cdot M\cap R(M^0;\form).
\end{equation}
Given $M,N\in R(M^0;\form)$ we write $M\degg^\ee N$ if $N\in \overline{\GL^\bullet(M^0,\form)\cdot M}$. 

Every $\ee$-representation can be written in an essentially unique way as a direct sum of indecomposable $\ee$-representations.  By \cite[Proposition~2.7]{DW}, the indecomposable $\ee$-representations of $(\stackrel{\rightarrow}{A_n},\sigma)$ are 
$$\{U_{i,j}\oplus U_{\sigma(j),\sigma(i)}|\, 1\leq i\leq j\leq n, i\leq \sigma(j)\}$$ 
in types $(\stackrel{\rightarrow}{A}_{\textrm{odd}},-1)$ and  $(\stackrel{\rightarrow}{A}_{\textrm{even}},+1)$, and 
$$\{U_{i,j}\oplus U_{\sigma(j),\sigma(i)}|\, 1\leq i\leq j\leq n, i< \sigma(j)\}\cup \{U_{i,\sigma(i)}|\, 1\leq i\leq \sigma(i)\leq n\}$$
in types $(\stackrel{\rightarrow}{A}_{\textrm{odd}},+1)$ and  $(\stackrel{\rightarrow}{A}_{\textrm{even}},-1)$. For this reason, the types $(\stackrel{\rightarrow}{A}_{\textrm{odd}},-1)$ and  $(\stackrel{\rightarrow}{A}_{\textrm{even}},+1)$ are called the \emph{split types} in \cite{BCI}.

\begin{lem}\label{Lem:AdjointEqui}
Let $M=(M^0,f)\in R(M^0;\form)$ be an $\ee$-representation of $\stackrel{\rightarrow}{A_n}$. Then, for every $1\leq i<j\leq n$, 
\begin{equation}\label{Eq:AdjointFij}
f_{i,j}^\star=(-1)^{j-i}f_{\sigma(j),\sigma(i)}.
\end{equation}
In particular, in the split types $\langle f_{i,\sigma(i)}(v), v\rangle=0$ for all $v\in M_i$.
\end{lem}
\begin{proof}
We proceed by induction on $j-i$. If $j-i=1$ then $f_{i,j}=f_i$ and 
$$f_{i,j}^\star=f_i^\star=-f_{\sigma(i+1)}=-f_{\sigma(j)}=-f_{\sigma(j),\sigma(i)}.$$ 
Suppose $j-i>1$ then for every $v\in M_i$ and $w\in M_{\sigma(j)}$ we have 
\begin{multline*}
\langle f_{i,j}(v),w\rangle=\langle f_{j-1}\circ f_{i, j-1}(v),w\rangle=-\langle f_{i,j-1}(v), f_{\sigma(j)}(w)\rangle=\\
=(-1)^{j-1-i}\langle v, f_{\sigma(j-1), \sigma(i)}\circ f_{\sigma(j)}(w)\rangle=-\langle v, (-1)^{j-i}f_{\sigma(j),\sigma(i)}(w)\rangle.
\end{multline*}

From \eqref{Eq:AdjointFij} we get that 
$$\langle f_{i,\sigma(i)}(v),v\rangle=(-1)^{\sigma(i)-i}\ee\langle f_{i,\sigma(i)}(v),v\rangle.$$ 
The number  $(\sigma(i)-i)=n+1-2i$ is even if $n$ is odd,  and  it is odd if $n$ is even. Therefore, in types $(A_\mathrm{odd}, -1)$ or $(A_\mathrm{even}, 1)$ we have $\langle f_{i,\sigma(i)}(v), v\rangle=-\langle f_{i,\sigma(i)}(v),v\rangle$ concluding the proof.
\end{proof}
The following consequence of Lemma~\ref{Lem:AdjointEqui} will be used later. To state it, we recall that for a given representation $M=(M^0,f)$, a sub-representation $N$ of $M$ is a graded vector sub-space 
$$N^0=\bigoplus N_i\subseteq \bigoplus M_i=M^0$$
such that $f_i(N_i)\subset N_{i+1}$, for every $1\leq i<n$. For $M\in R(M^0;\form)$, a sub-representation $N\subseteq M$ is called \emph{isotropic} if $N^0$ is an isotropic subspace of $M^0$ with respect to the non-degenerate form $\form$ (i.e. $N^0\subseteq (N^0)^\perp$).
\begin{cor}\label{Cor:SplitIndIsotropic}
In the split types every indecomposable sub-representation of an $\ee$-representation is isotropic. 
\end{cor}
\begin{proof}
Let $M=((M_i)_{i=1}^n,(f_i:M_i\rightarrow M_{i+1})_{i=1}^{n-1})$ be an $\ee$-representation and let $U_{k,\ell}\subset M$ be an indecomposable sub-representation. Let $v,w\in U_{k,\ell}$, say $v\in M_i$ and $w\in M_j$ for $i<j$. Then, up to a scalar, $w=f_{i,j}(v)$. If $j\neq \sigma(i)$ then $\langle v,w\rangle=0$ by definition of the $\ee$-form; if $j=\sigma(i)$ then $\langle v,w\rangle=0$ by Lemma~\ref{Lem:AdjointEqui}.
\end{proof}

We now give a handy criterion to recognize if an $\stackrel{\rightarrow}{A_n}$-representation is an $\ee$-representation in terms of its rank sequence.

\begin{thm}\label{Thm:MainEquioriented}
Let $M$ be an $\stackrel{\rightarrow}{A_n}$-representation. Then $M$ admits the structure of an $\ee$-representation if and only if its rank sequence $r^M=(r_{i,j}^M)$ satisfies the following relations: 
\begin{eqnarray}\label{Eq:RankErepEquioriented1}
r_{i,j}^M=r_{\sigma(j),\sigma(i)}^M;
\end{eqnarray}
\begin{multline}\label{Eq:RankErepEquioriented2}
\text{in type }(\stackrel{\rightarrow}{A}_{\mathrm{odd}},-1)\text{ and in type }(\stackrel{\rightarrow}{A}_{\mathrm{even}},+1), r_{i,\sigma(i)}^M\text{ is even for any $i\leq \sigma(i)$}.
\end{multline}
\end{thm}

\begin{proof}
We write $M=\bigoplus U_{i,j}^{m_{i,j}}$ as a direct sum of indecomposable representations. Thus, $M$ is an $\ee$-representation if and only if it is a direct sum of indecomposable $\ee$-representations. The result follows from \eqref{Eq:MultiplicityRank} and the description of the indecomposable $\ee$-representations.
\end{proof}

\subsection{(Symmetric) degenerations}
In this section, we describe the degeneration order $\degg^\ee$ in the split types $(\stackrel{\rightarrow}{A}_{\mathrm{odd}},-1)$ and $(\stackrel{\rightarrow}{A}_{\mathrm{even}},+1)$. 

Let $M$ be an $\ee$-representation of $(\stackrel{\rightarrow}{A}_{n},\sigma)$. Since we are in a split type, $M$ is the union of segments $[i,j]\cup [\sigma(j),\sigma(i)]\subset M$. We call $[i,j]\cup [\sigma(j),\sigma(i)]$ a \emph{symmetric segment} of $M$.
We notice that the segments $[i,\sigma(i)]$ appear with even multiplicity. 
\begin{definition}\label{Def:SymmCutsShifts}
Let $M$ be an $\ee$-representation. 
\begin{enumerate}
\item A \emph{symmetric cut} at a column $q$ of a symmetric segment $[t,s]\cup [\sigma(s), \sigma(t)]\subset M$ is the cut at column $q$ of $[t,s]$ followed by a cut at $\sigma(q+1)$ of $[\sigma(s),\sigma(t)]$, for every $1\leq t\leq q< s\leq n$ (see Table~\ref{Table:SymmetricCut}). These two cuts are called dual to each other.
\item A \emph{symmetric shift} of the segment  $[r,s]$ from the symmetric segment $[t,s]\cup [\sigma(s),\sigma(t)]\subset M$ to the symmetric segment $[q,r]\cup [\sigma(r),\sigma(q)]\subset M$ is the shift of $[r,s]$ from $[t,s]$ to $[q,r]$ together with the shift of $[\sigma(q),\sigma(t)]$ from  $[\sigma(s),\sigma(t)]\subset M$ to  $[\sigma(r),\sigma(q)]\subset M$ (see Table~\ref{Table:SymmetricShift}). These two shifts are called dual to each other.
\end{enumerate}
Given $M,N\in R(M^0;\form)$ we write $M\leq_c^\ee N$ if $N$ is obtained from $M$ by a finite sequence of symmetric cuts or symmetric shifts.
\end{definition}
\begin{table}[htbp]
\begin{tabular}{|c|c|}
\hline
$
\xymatrix@R=0pt@C=4pt{
t&q&&&s&&&\\
*{\bullet}\ar@{-}[rrrr]&*{\bullet}\ar[r]&*{\bullet}&&*{\bullet}&&&\\
&&&*{\bullet}\ar@{-}[rrrr]&&*{\bullet}\ar[r]&*{\bullet}&*{\bullet}\\
&&&\sigma(s)&&&\sigma(q)&\sigma(t)
}
$
&
$
\xymatrix@R=0pt@C=1pt{
t&q&q+1&&s&&&\\
*{\bullet}\ar@{-}[r]&*{\bullet}&*{\bullet}\ar@{-}[rr]&&*{\bullet}&&&\\
&&&*{\bullet}\ar@{-}[rr]&&*{\bullet}&*{\bullet}\ar@{-}[r]&*{\bullet}\\
&&&\sigma(s)&&\sigma(q+1)&\sigma(q)&\sigma(t)
}
$
\\\hline
$M$&$N$\\\hline
\end{tabular}
\caption{A symmetric cut from $M$ to $N$ at columns $q$ of the symmetric segment $[t,s]\cup [\sigma(s),\sigma(t)]$}
\label{Table:SymmetricCut}
\end{table}
\begin{table}[htbp]
\begin{tabular}{|c|c|}
\hline
$
\xymatrix@R=0pt@C=4pt{
t&q&r&&s&&&\\
*{\bullet}\ar@{-}[rrrr]&*{\bullet}&*{\bullet}&&*{\bullet}&&&\\
&*{\bullet}\ar@{-}[r]&*{\bullet}&&&&&\\
&&&*{\bullet}\ar@{-}[rrrr]&&*{\bullet}&*{\bullet}&*{\bullet}\\
&&&&&*{\bullet}\ar@{-}[r]&*{\bullet}&\\
&&&\sigma(s)&&\sigma(r)&\sigma(q)&\sigma(t)
}
$
&
$
\xymatrix@R=0pt@C=4pt{
t&q&r&&s&&&\\
*{\bullet}\ar@{-}[rr]&*{\bullet}&*{\bullet}&&&&&\\
&*{\bullet}\ar@{-}[rrr]&*{\bullet}&&*{\bullet}&&&\\
&&&*{\bullet}\ar@{-}[rrr]&&*{\bullet}&*{\bullet}&\\
&&&&&*{\bullet}\ar@{-}[rr]&*{\bullet}&*{\bullet}\\
&&&\sigma(s)&&\sigma(r)&\sigma(q)&\sigma(t)
}
$
\\\hline
$M$&$N$\\\hline
\end{tabular}
\caption{A symmetric shift from $M$ to $N$ of the segment $[r,s]\cup [\sigma(q),\sigma(t)]$ from $[t,s]\cup [\sigma(s),\sigma(t)]\subset M$ to $[q,r]\cup [\sigma(r),\sigma(q)]\subset M$}
\label{Table:SymmetricShift}
\end{table}

In this section we prove the following result:
\begin{thm}\label{Thm:ee-degenerationRanks}
In the split types, $M\degg^\ee N$ if and only if $M\leq_c^\ee N$  if and only if $r^M\geq r^N$.
\end{thm}

\subsection{Proof of Theorem~\ref{Thm:ee-degenerationRanks}}
%
The proof of Theorem~\ref{Thm:ee-degenerationRanks} is illustrated as follows:
$$
\xymatrix@C=50pt{
M\degg^\ee N\ar@/^8pt/^{\text{upper-semicontinuity}}@{=>}[rr]&M\leq_c^\ee N\ar@{=>}^{\text{Lemma}~\ref{Lem:CimpliesDeg}}[l]
&r^M\geq r^N\ar@{<=>}^{\text{Lemma}~\ref{Lem:CEquivalentR}}[l]}
$$

\begin{lem}\label{Lem:CimpliesDeg}
Let $M,N\in R(M^0;\form)$. If $M\leq_c^\ee N$ then $M\degg^\ee N$.
\end{lem}

\begin{proof}
A cut at columns $q$ of the segment $[t,s]\subset M$ corresponds to the one-parameter subgroup of $\GL^\bullet(M^0)$ induced by the non-split short exact sequence
$$
\xymatrix{
0\ar[r]&U_{q+1, s}\ar[r]&U_{t,s}\ar[r]&U_{t,q}\ar[r]&0
}
$$
Since we are in a split type, by Corollary~\ref{Cor:SplitIndIsotropic}, $U_{q+1,s}$ is an isotropic subrepresentation of $M$. By \cite[Theorem~3.1]{BCI} we have $M\degg^\ee U_{q+1,s}\oplus U_{\sigma(s),\sigma(q+1)}\oplus U_{q+1,s}^\perp/U_{q+1,s}=N$. 

A shift of $[q,r]\subset M$ from $[t,s]\subset M$ to $[q,s]\subset N$ corresponds to the one-parameter subgroup of $\GL^\bullet(M^0)$ induced by the non-split short exact sequence
$$
\xymatrix{
0\ar[r]&U_{q, s}\ar[r]&U_{t,s}\oplus U_{q,r}\ar[r]&U_{t,r}\ar[r]&0
}
$$
Since we are in a split type, by Corollary~\ref{Cor:SplitIndIsotropic}, $U_{q,s}$ is an isotropic subrepresentation of $M$. By \cite[Theorem~3.1]{BCI} we have $M\degg^\ee U_{q,s}\oplus U_{\sigma(s),\sigma(q)}\oplus U_{q,s}^\perp/U_{q,s}=N$. 
\end{proof}

\begin{prop}\label{Prop:RankSequenceOrthogonal}
In the split types, let $M$ be an $\ee$-representation. Let $L=P_q=U_{q,n}$, suppose that $L$ embeds into $M$ and let $\iota:L\hookrightarrow M$ be a generic embedding. Then $\iota(L)$ is isotropic and the rank sequence of $M(1):=\iota(L)^\perp/\iota(L)$ is 
\begin{equation}\label{Eq:RankPerpQuotient}
r_{k,\ell}^{M(1)}=\left\{
\begin{array}{ll}
r_{k,\ell}^M-2&\text{if }r_{q,\ell}^M\leq r_{k,\ell}^M\text{ and }r_{k,\sigma(q)}^M\leq r_{k,\ell}^M,\\
r_{k,\ell}^M-1&\text{if }r_{q,\ell}^M\leq r_{k,\ell}^M\text{ or }r_{k,\sigma(q)}^M\leq r_{k,\ell}^M\text{ but not both},\\
r_{k,\ell}^M&\text{otherwise,}
\end{array}
\right.
\end{equation}
with the convention that $r^M_{s,t}=\infty$ if $s>t$.
In particular, $M\leq_c^\ee L\oplus\nabla L\oplus M(1)$.
\end{prop}
\begin{proof}
The proof of this result is similar to the proof of Proposition~\ref{Prop:RankSequenceQuotient} and it uses Corollary~\ref{Cor:Iso}.
\end{proof}
\begin{rem}
The generic quotient of $\iota(L)^\perp$ by $L$ is not necessarily $\iota(L)^\perp/\iota(L)$ (see \cite[Section~6.4]{BCI}).
\end{rem}

%
\begin{lem}\label{Lem:CEquivalentR}
Let $M,N\in R(M^0;\form)$. Then $M\leq_c^\ee N$ if and only if $r^M\geq r^N$.
\end{lem}
\begin{proof}
By Lemma~\ref{Lem:RankCutShift}, $M\leq_c^\ee N$ implies $r^M\geq r^N$. To prove the converse suppose that $r^M\geq r^N$. If $r^M=r^N$ then in view of \eqref{Eq:MWZTheorem}, $M$ and $N$ are isomorphic as $\ee$-representations and hence there is nothing to prove. Let us assume that $r^M> r^N$. Without loss of generality we can assume that $\mathrm{dim}\,M_n=r_{n,n}^M=r_{n,n}^N>0$, because otherwise both $M$ and $N$ are supported on a smaller quiver of the same type. Let $i$ be the maximal index such that $r_{i,n}^N>r_{i-1,n}^N$. By Lemma~\ref{Lem:Embeddings}, $P_i=U_{i,n}$ is a direct summand of $N$ (recall our convention that $r_{0,n}^N=0=r_{i,n+1}$). For simplicity of notation we put $L:=P_i$.  We decompose $N$ as $N=L\oplus \nabla L\oplus N(1)$ for a suitable $\ee$-representation $N(1)$. Since $r_{i,n}^M\geq r_{i,n}^N>0$, by Lemma~\ref{Lem:Embeddings}, $P_i$ embeds into $M$. By Corollary~\ref{Cor:SplitIndIsotropic}, every embedding of $L$ into $M$ is isotropic. Let $\iota:L\rightarrow M$ be a generic (isotropic) embedding of $L$ into $M$. Let $M(1):=\iota(L)^\perp/\iota(L)$. By Proposition~\ref{Prop:RankSequenceOrthogonal}, $M\leq_c^\ee L\oplus\nabla L\oplus M(1)$. We claim that $r^{M(1)}\geq r^{N(1)}$. The rank sequence $r^{M(1)}$ of $M(1)$ is given by \eqref{Eq:RankPerpQuotient}. Since $L^\perp\degg L\oplus N(1)$ we have $r^{L^\perp}_{k,\ell}\geq r^{L\oplus N(1)}_{k,\ell}$, for every $1\leq k\leq \ell\leq n$; moreover 
\begin{equation}\label{Eq:RankLoplusX}
r_{k,\ell}^{L\oplus N(1)}=\left\{\begin{array}{ll}r_{k,\ell}^{N(1)}+1&\text{if }i\leq k\leq\ell\leq j;\\&\\r_{k,\ell}^{N(1)}&\text{otherwise.}\end{array}\right.
\end{equation}
For $i\leq k\leq \ell\leq j$, we have 
$$
r_{i,\ell}^M-r_{i,j+1}^M=\dim(\im\,f_{i,\ell}\cap \Ker\,f_{\ell,j+1})\leq \dim (\im\,f_{k,\ell}\cap \Ker\,f_{\ell,j+1})=r_{k,\ell}^M-r_{k,j+1}^M.
$$
Combining \eqref{Eq:RankPerpQuotient} with \eqref{Eq:RankLoplusX} we get for $i\leq k\leq \ell\leq j$
$$
r_{k,\ell}^{M(1)}=r_{k,\ell}^{L^\perp}-1\geq r_{k,\ell}^{L\oplus N(1)}-1=r_{k,\ell}^{N(1)}.
$$
By induction (on the rank difference) we get $M(1)\leq_c^\ee N(1)$. Thus, $M\leq_c^\ee L\oplus \nabla L\oplus M(1)\leq_c^\ee L\oplus\nabla L\oplus N(1)=N$.
\end{proof}

\section{Linear degenerate symplectic flag varieties: PBW locus}\label{Sec:PBW-locus}

For a positive integer $n$, set $[n]=\{1,2,\ldots,n\}$. In this section we fix a subset $\bi=\{i_1,\ldots,i_t\}\subseteq [n-1]$ with $1\leq i_1<\ldots<i_t\leq n-1$. For convenience we extend the notation by requiring $i_0=0$.

\subsection{Lagrangian quiver Grassmannian}

We will define Lagrangian quiver Grassmannian in the special case $(\stackrel{\rightarrow}{A}_{\textrm{odd}},-1)$. Thus, in this section we work with the quiver $\stackrel{\longrightarrow}{A_{2n-1}}$ with the following labelling of vertices and arrows: 
$$1\stackrel{\alpha_1}{\longrightarrow} 2\stackrel{\alpha_2}{\longrightarrow} \cdots \longrightarrow n-1\stackrel{\alpha_{n-1}}{\longrightarrow}\omega \stackrel{\alpha_{n-1}^*}{\longrightarrow} (n-1)^*\longrightarrow\cdots\stackrel{\alpha_1^*}{\longrightarrow} 1^*$$
General theory of Lagrangian quiver Grassmannians will be developed in a separate publication.

Let $M\in R(M^0;\form)$ be a symplectic $\stackrel{\longrightarrow}{A_{2n-1}}$-representation. A subrepresentation $N\subseteq M$ is called \emph{Lagrangian}, if the vector subspace 
$$N^0=\bigoplus N_i\subseteq \bigoplus M_i=M^0$$ 
is Lagrangian, i.e. if $N^0=(N^0)^\perp$ where $(N^0)^\perp$ denotes the orthogonal of $N^0$ with respect to the symplectic form $\form$. By identifying $M^0$ with its dual via the non-degenerate bilinear form, we see that $\nabla(M^0/N^0)$ is the annihilator of $N^0$ and thus, $N\subset M$ is Lagrangian if and only if   
$M/N\cong \nabla N$. We notice that if $N\subset M$ is Lagrangian then, since $\omega=\sigma(\omega)$, $N_\omega\subset M_\omega$ is a Lagrangian subspace of the symplectic vector space $(M_\omega,\form|_{M_\omega\times M_\omega})$.

\begin{definition}
Let $\be=(e_1,\ldots,e_{1^*})\in\mathbb{N}^{2n-1}$ be a dimension vector. The Lagrangian quiver Grassmannian $\mathrm{Lag}_\be(M)$ for a symplectic $\stackrel{\longrightarrow}{A_{2n-1}}$-representation $M$ and a dimension vector $\be$ is defined to be the set of Lagrangian submodules of $M$ of dimension vector $\be$.
\end{definition}

The Lagrangian quiver Grassmannian is a closed subset of the quiver Grassmannian $\mathrm{Gr}_\be(M)$ and is embedded into the product of Grassmannians $\prod_{i\in Q_0}\mathrm{Gr}_{e_i}(M_i)$. One can show that if $M$ is  projective, then $\mathrm{Lag}_{\underline{\mathrm{dim}}M}(M\oplus\nabla(M))$ is irreducible and reduced. Since this result is not necessary in the paper, we omit its proof.

\subsection{Abelianizations of symplectic flag varieties}

Let $\g=\mathfrak{sp}_{2n}(\mathbb{C})$ be the Lie algebra of type $C_n$. Fix a triangular decomposition $\g = \n_+ \oplus \h \oplus \n_-$ of $\g$. Denote by $\Phi$ the root system of $\g$, $\Phi^+\subseteq\Phi$ the set of positive roots, $\{\alpha_1,\ldots,\alpha_n\}$ the set of simple roots, $\{\alpha_1^\vee,\ldots,\alpha_n^\vee\}$ the set of simple coroots, $\Lambda^+$ the set of dominant integral weights and $\varpi_1,\ldots,\varpi_n$ the fundamental weights.

The set of positive roots $\Phi^+$ consists of $\alpha_{i,j}:=\alpha_i+\ldots+\alpha_j$ for $1 \leq i\leq j\leq n$ and $\alpha_{i,\overline{j}}:=\alpha_{i,n}+\alpha_{j,n-1}$ for $1\leq i\leq j\leq n-1$. For $\beta\in \Phi^+$ we fix a non-zero generator $f_\beta$, called a root vector, in the weight space of $\n_-$ of weight $-\beta$. 

We are going to define a Lie algebra filtration on $\n_-$ and an algebra filtration on $U(\n_-)$. For $\bd\in \mathbb{R}^{\Phi^+}$ denote $d_\beta:=\bd(\beta)$. Consider the following subspaces of $\n_-$ and $U(\n_-)$: for $k\in\mathbb{N}$,
$$(\n_{-})^{\mathbf{d}}_{\leq k}=\langle f_\beta\mid\, \beta\in\Phi^+,\  d_\beta\leq k\rangle_{\mathbb{C}},$$
$$U(\n_{-})^{\mathbf{d}}_{\leq k}=\langle f_{\beta_1}\cdots f_{\beta_\ell}\mid \beta_1,\ldots,\beta_\ell\in\Phi^+, d_{\beta_1}+\ldots+d_{\beta_\ell}\leq k \rangle_{\mathbb{C}}.$$
In general they do not give an $\mathbb{N}$-filtration of Lie algebra and an $\mathbb{N}$-filtration of algebra. We will impose more assumptions on the function $\bd$. For simplicity we will write $d_{i,j}:=d_{\alpha_{i,j}}$ and $d_{i,\overline{j}}:=d_{\alpha_{i,\overline{j}}}$. Fix the following total order on the indices:
$$1<2<\ldots<n<\overline{n-1}<\ldots<\overline{1}.$$

Let $F^{\bi}\subseteq\mathbb{R}^{\Phi^+}$ be a polyhedral cone defined by the following equalities and inequalities: for 
\begin{enumerate}
\item[-] for $1\leq i\leq j\leq n-1$, $j+1\leq \ell\leq \overline{i}$ with $j\in\bi$, $d_{i,j}+d_{j+1,\ell}\geq d_{i,\ell}$;
\item[-] for $1\leq i\leq \ell\leq j+1\leq n$ with $j\in\bi$, $d_{i,j}+d_{\ell,\overline{j+1}}\geq d_{i,\overline{\ell}}$;
\item[-] for any other two roots $\beta_1,\beta_2\in\Phi^+$ being not in the above two cases and satisfying $\beta_1+\beta_2\in\Phi^+$, $d_{\beta_1}+d_{\beta_2}=d_{\beta_1+\beta_2}$.
\item[-] for $1\leq i< j\le k <\ell\leq n$ , $d_{i,k}+d_{j,\ell} = d_{i,\ell} + d_{j,k}$;
\item[-] for $1\leq i< j\le k ,\ell\leq n$ , $d_{i,\overline k}+d_{j,\ell} = d_{i,\ell} + d_{j,\overline k}$;
\item[-] for $1\leq i< j< k <\ell\leq n-1$ , $d_{i,\overline j}+d_{k,\overline \ell} = d_{i,\overline k} + d_{j,\overline \ell} = d_{i,\overline \ell}+d_{j,\overline k}$.
\end{enumerate}

Such polyhedral cones are termed \emph{Dynkin faces} in \cite{EFFS}. It is shown in \textit{loc.cit.} that such faces are nonempty.

Let $\mathrm{relint}(F^\bi)$ denote the relative interior of the cone $F^\bi$. The following results are proved in \textit{loc.cit.}:

\begin{prop}\label{Prop:Prelim}
Fix $\bd\in F^\bi$, the following statements hold:
\begin{enumerate}
\item The subspaces $\{(\n_{-})^{\mathbf{d}}_{\leq k}\mid k\in\mathbb{N}\}$ defines a Lie algebra filtration on $\n_-$ and the subspaces $\{U(\n_{-})^{\mathbf{d}}_{\leq k}\mid k\in\mathbb{N}\}$ defines an algebra filtration on $U(\n_-)$. 
\item Let $\n_-^\bd$ and $U(\n_-)^\bd$ denote the associated graded Lie algebra and the associated graded algebra associated to the filtrations above. Then $U(\n_-^\bd)\cong U(\n_-)^\bd$.
\item For any $\bd,\bd'\in\mathrm{relint}(F^\bi)$, $\n_-^\bd\cong\n_-^{\bd'}$ as Lie algebras and $U(\n_-^\bd)\cong U(\n_-^{\bd'})$ as algebras.
\end{enumerate}
\end{prop}

For $\lambda\in\Lambda^+$, let $V(\lambda)$ be the finite dimensional irreducible representation of $\g$ of highest weight $\lambda$. Fix a highest weight vector $v_\lambda\in V(\lambda)$, then $V(\lambda)=U(\n_-)\cdot v_\lambda$. The algebra filtration on $U(\n_-)$ induces a filtration on $V(\lambda)$ by defining 
$$(V(\lambda))_{\leq k}^{\bd}=U(\n_-)^{\bd}_{\leq k}\cdot v_\lambda.$$
Let $V^\bd(\lambda)$ denote the associated graded vector space. It carries naturally a graded $U(\n_-^\bd)$-module structure which is cyclic, that is to say, $V^\bd(\lambda)=U(\n_-^\bd)\cdot v_\lambda^\bd$ where $v_\lambda^\bd$ is the image of $v_\lambda$ in $V^\bd(\lambda)$. 

Let $N_-^\bd:=\exp(\n_-^\bd)$ denote the connected Lie group with Lie algebra $\n_-^\bd$. 

\begin{definition}
For $\lambda\in\Lambda^+$, we define the $\bd$-degenerate flag variety associated to $\lambda$ by:
$$\mathcal{F}^{\bd}(\lambda):=\overline{N^{\bd}_-\cdot [v_\lambda^\bd]}\hookrightarrow \mathbb{P}(V^{\bd}(\lambda)),$$
where $[v_\lambda^\bd]$ is the highest weight line generated by $v_\lambda^\bd$ in the projective space $ \mathbb{P}(V^{\bd}(\lambda))$. 
\end{definition}

It is shown in \cite{EFFS} that for any $\bd,\be\in\mathrm{relint}(F^\bi)$, the projective varieties $\mathcal{F}^{\bd}(\lambda)$ and $\mathcal{F}^{\be}(\lambda)$ are isomorphic.

\subsection{Schubert varieties}

To the subset $\bi=\{i_1,\ldots,i_t\}$ which is fixed in the beginning of this section, we associate a Schubert variety for the symplectic group of higher rank. 

Let $\wt{\g}=\mathfrak{sp}_{2(n+t)}$ be the symplectic Lie algebra of type $C_{n+t}$ with a triangular decomposition $\wt{\g}=\wt{\n}_+\oplus\wt{\mathfrak{h}}\oplus\wt{\n}_-$. The set of positive roots will be denoted by $\wt{\Phi}^+$ and the simple roots are $\wt{\alpha}_1$, $\ldots$, $\wt{\alpha}_{n+t}$. The weight lattice is $\wt{\Lambda}$, the set of dominant integral weights is $\wt{\Lambda}^+$ and the fundamental weights are $\wt{\varpi}_1$, $\ldots$, $\wt{\varpi}_{n+t}$. Let $\wt{G}$ be the connected simply connected algebraic group of Lie algebra $\wt{\g}$, $\wt{B}$ be the Borel subgroup in $\wt{G}$ of Lie algebra $\mathfrak{b}_+:=\wt{\n}_+\oplus \wt{\mathfrak{h}}$, $W(\wt{\g})$ be the Weyl group of $\wt{\g}$ of simple reflections $s_1,\ldots,s_{n+t}$.

We will consider the following Weyl group element in $W(\wt{\g})$:
\begin{equation}\label{Eq:WGElement}
w_\bi=s_{n+t}(s_{n+t-1}s_{n+t})\cdots (s_{t+1}\cdots s_{n+t})v_t\cdots v_1
\end{equation}
where 
$$v_k:=(s_k\cdots s_{i_k+k-1})\cdots (s_k\cdots s_{i_{k-1}+k}).$$ 
It is shown in \cite{EFFS} that these decompositions of $w_\bi$ and $v_k$ are reduced.

For $\wt{\lambda}\in\wt{\Lambda}^+$ let $\wt{V}(\wt{\lambda})$ be the irreducible representation of $\wt{\g}$ of highest weight $\wt{\lambda}$. Let $v_{w_\bi(\wt{\lambda})}:=w_\bi(v_{\wt{\lambda}})$ be an extremal weight vector in $\wt{V}(\wt{\lambda})$. The Demazure module $\wt{V}_{w_\bi}(\wt{\lambda})$ is defined as $U(\mathfrak{b}_+)\cdot v_{w_\bi(\wt{\lambda})}$.

The Schubert variety associated to $w_\bi$ is defined to be 
$$X_{w_\bi}:=\overline{\wt{B}\cdot v_{w_\bi(\wt{\lambda})}}\hookrightarrow \wt{G}/\wt{P}_{\wt{\lambda}},$$
where $\wt{P}_{\wt{\lambda}}$ is the parabolic subgroup in $\wt{G}$ stabilizing $\wt{\lambda}$.

\subsection{Linear degenerate symplectic flag varieties: PBW locus}

We fix on $V:=\mathbb{C}^{2n}$ a symplectic form $\omega_V(-,-)$ and let $\mathrm{Gr}_i(V)$ be the Grassmann variety of subspaces of dimension $i$ in $V$. Fix a basis $v_1,\ldots,v_n,v_{n+1},\ldots,v_{2n}$ of $V$ satisfying $\omega_V(v_i,v_{2n+1-j})=\delta_{i,j}$ for $1\leq i, j\leq n$. For a subset $K\subset [2n]$, let $\mathrm{pr}_K:V\to V$ be the linear map defined by
$$\mathrm{pr}_K\left(\sum_{i\in [2n]}\lambda_iv_i\right)=\sum_{i\in [2n]\setminus K}\lambda_iv_i$$

For the subset $\bi$ fixed in the beginning of this section, we define linear maps $f_1,\ldots,f_{n-1}$ by
$$f_i:=\begin{cases}
\mathrm{pr}_{i_k+1}, & \text{if $i=i_k$ for some $1\leq k\leq t$};\\
\mathrm{id}, & \text{otherwise}.
\end{cases}$$

\begin{definition}\label{Def:LDSFV}
The associated linear degenerate symplectic flag variety is defined by
\begin{eqnarray*}
\mathrm{Sp}\mathcal{F}_{2n}^\bi:=\{(V_1,\ldots,V_n)\in\prod_{i=1}^n\mathrm{Gr}_i(V)\mid\ f_i(V_i)&\subseteq& V_{i+1}\text{ for $1\leq i\leq n-1$}\\
& &\text{and $V_n$ is Lagrangian}\}.
\end{eqnarray*}
\end{definition}

By definition it is a closed subset in the product of Grassmann varieties, it is not yet clear whether $\mathrm{Sp}\mathcal{F}_{2n}^\bi$ is irreducible or reduced, so we consider the reduced structure on it.

\subsection{Main result on PBW locus}

For the fixed $\bi\subseteq [n-1]$ consider the set $[n+t]\setminus\{i_1+1,i_2+2,\ldots,i_t+t\}$ and order the elements in the set as $j_1<\ldots<j_n$. Define a map $\sigma_\bi:[n]\to [n+t]$ sending $1\leq k\leq n$ to $j_k$ and define a morphism of monoids $\Psi:\Lambda^+\to\wt{\Lambda}^+$ by:
$$\Psi(\lambda):=\sum_{j=1}^n(\lambda,\alpha_{j}^\vee) \wt{\varpi}_{\sigma_\bi(j)}.$$

Let 
$$M^\bi:=P_1^{\oplus n-t}\oplus\bigoplus_{k=1}^t P_{\sigma(i_k)}\oplus \bigoplus_{k=1}^t P_{i_k+1}$$
be a projective $\stackrel{\longrightarrow}{A_{2n-1}}$-representation and $\be=(1,2,\ldots,2n-1)\in\mathbb{N}^{2n-1}$ be a dimension vector. 

\begin{thm}\label{Thm:Iso}
The following varieties are isomorphic:
\begin{enumerate}
\item[(1)] the linear degenerate flag variety $\mathrm{Sp}\mathcal{F}_{2n}^\bi$;
\item[(2)] the Lagrangian quiver Grassmannian $\mathrm{Lag}_{\be}(M^\bi\oplus\nabla (M^\bi))$;
\item[(3)] the Schubert variety $X_{w_\bi}$;
\item[(4)] for any $\bd\in \mathrm{relint}(F^\bi)$, the $\bd$-degenerate flag variety $\mathcal{F}^\bd(\rho)$, where $\rho=\varpi_1+\ldots+\varpi_n$ is the sum of all fundamental weights.
\end{enumerate}
\end{thm}

\begin{cor}\label{Cor:Iso}
\begin{enumerate}
\item[(1)] The linear degenerate flag variety $\mathrm{Sp}\mathcal{F}_{2n}^\bi$ and the Lagrangian quiver Grassmannian $\mathrm{Lag}_{\be}(M^\bi\oplus\nabla (M^\bi))$ are irreducible and reduced.
\item[(2)] The linear degenerate flag variety $\mathrm{Sp}\mathcal{F}_{2n}^\bi$ and the Lagrangian quiver Grassmannian $\mathrm{Lag}_{\be}(M^\bi\oplus\nabla (M^\bi))$ are normal, Cohen-Macaulay, Frobenius splitting and have rational singularities.
\end{enumerate}
\end{cor}

These results are proved for $\mathrm{Sp}\mathcal{F}_{2n}^\bi$ where $\bi=[n-1]$ in \cite{FeFiL,CIL}.

\section{Proof of Theorem \ref{Thm:Iso}}\label{Sec:Proof}

We prove Theorem \ref{Thm:Iso} in this section. Corollary \ref{Cor:Iso} holds by the equivalence of the first three points in the theorem.

\subsection{\texorpdfstring{$(3)\Leftrightarrow (4)$}{ 3 equivalent 4 }}
It has been proved in \cite[Corollary 3.3]{EFFS} that $X_{w_\bi}$ is isomorphic to $\mathcal{F}^\bd(\rho)$. The proof of \cite[Corollary 3.3]{EFFS} uses a careful study of the defining ideal of the Demazure module to obtain a surjective map $\wt{V}_{w_\bi}(\Psi(\rho))\to V^\bd(\rho)$ which is in fact not necessary in the symplectic case. 
Indeed, this statement can also be proved directly by adapting the proof in \cite[Section 5]{CFFFR}: the word $w_\bi$ can be obtained from folding a triangular element (see \cite{BFK} for the definition of such elements in type $A$ and $C$), hence it is again a triangular element. 
By the results in \cite{BFK}, \cite[Proposition 10]{CFFFR} holds in the symplectic case, hence the surjective map $V^\bd(\rho)\to\wt{V}_{w_\bi}(\Psi(\rho))$, obtained from studying its image in the Cartan component of the tensor product of fundamental modules, is also injective.

\subsection{\texorpdfstring{$(1)\Leftrightarrow (3)$}{1 equivalent 3 }}
The proof is executed in several steps. The strategy of the proof is the same as \cite{CIL}. However, the combinatorics is much more involved so we decide to provide necessary details of the proof.

Recall that $V=\mathbb{C}^{2n}$ is endowed with a symplectic form $\omega_V$. Recall that we have fixed a symplectic basis $v_1,\ldots,v_{2n}$ of $V$ with the following property: for any $1\leq i,j\leq n$,
$$\omega_V(v_i,v_{2n+1-j})=\delta_{i,j}.$$

\subsubsection{Linear degenerate symplectic flag variety as invariant variety}

In this part we realize the linear degenerate symplectic flag variety as an invariant variety of a linear degenerate flag variety \cite{CFFFR,CFFFR20} under an involution.

For this, define 
$$\bi':=\{i_1,\ldots,i_t,2n-1-i_t,\ldots,2n-1-i_1\}\subseteq [2n-1]$$ 
where we write $i_{t+\ell}:=2n-1-i_\ell$ for $1\leq \ell\leq t$, and the associated linear degenerate flag variety \cite{CFFFR,CFFFR20} is defined by:
$$\mathcal{F}_{2n}^{\bi'}:=\left\{(V_1,\ldots,V_{2n-1})\in\prod_{k=1}^{2n-1}\mathrm{Gr}_k(V)\mid f_k'(V_k)\subseteq V_{k+1}\text{ for $1\leq k\leq 2n-2$}\right\},$$
where the linear maps $f_1',\ldots,f_{2n-2}'$ are defined by
$$f_i':=\begin{cases}
\mathrm{pr}_{i_k+1}, & \text{if $i=i_k$ for some $1\leq k\leq 2t$};\\
\mathrm{id}, & \text{otherwise}.
\end{cases}$$

We consider the following well-defined involution 
$$\tau:\prod_{k=1}^{2n-1}\mathrm{Gr}_k(V)\to \prod_{k=1}^{2n-1}\mathrm{Gr}_k(V),\ \ (V_1,\ldots,V_{2n-1})\mapsto (V_{2n-1}^\bot,\ldots,V_1^\bot),$$
where the orthogonal spaces are with respect to the symplectic form $\omega_V$. Since $\mathcal{F}_{2n}^{\bi'}$ is a subset of the product of Grassmann varieties, let $(\mathcal{F}_{2n}^{\bi'})^\tau$ denote the invariant space with respect to this action of $\tau$.

\begin{lem}\label{Lem:SpInv}
There exists an isomorphism $\mathrm{Sp}\mathcal{F}_{2n}^\bi\cong (\mathcal{F}_{2n}^{\bi'})^\tau$.
\end{lem}

\begin{proof}
Define two morphisms
$$\phi:\mathrm{Sp}\mathcal{F}_{2n}^\bi\to \mathcal{F}_{2n}^{\bi'},\ \ (V_1,\ldots,V_n)\mapsto (V_1,\ldots,V_n,V_{n-1}^\bot,\ldots,V_1^\bot);$$
$$\psi:(\mathcal{F}_{2n}^{\bi'})^\tau\to\mathrm{Sp}\mathcal{F}_{2n}^\bi,\ \ (V_1,\ldots,V_{2n-1})\mapsto (V_1,\ldots,V_n).$$
It remains to show that both maps are well-defined, and the image of $\phi$ is $\tau$-invariant.

\begin{enumerate}
\item The image of $\phi$ is in $\mathcal{F}_{2n}^{\bi'}$. Indeed, since $V_n=V_n^\bot$, it suffices to show that $f_k'(V_{2n-k}^\bot)\subseteq V_{2n-k-1}^{\bot}$ for $n\leq k\leq 2n-2$. There is nothing to show when $f_k'=\mathrm{id}$; otherwise since $n\leq k\leq 2n-2$, we can write $k=i_{t+\ell}$ for $1\leq\ell\leq t$, and $f_k'=\mathrm{pr}_{i_{t+\ell}+1}$. To show that $\mathrm{pr}_{i_{t+\ell}+1}(V_{i_\ell+1}^\bot)\subseteq V_{i_\ell}^\bot$, it suffices to show that $\mathrm{pr}_{i_\ell+1}^*=\mathrm{pr}_{2n-i_\ell}$ where $\mathrm{pr}_{i_\ell+1}^*$ is the adjoint map with respect to the symplectic form. This last statement holds according to the choice of the basis $v_1,\ldots,v_{2n}$.
\item The image of $\phi$ is $\tau$-invariant: it follows from the assumption that $V_n$ is Lagrangian and hence $V_n^\bot=V_n$.
\item The map $\psi$ is well-defined. This follows from the definition of $\bi'$, and the fact that if $(V_1,\ldots,V_{2n-1})$ is $\tau$-invariant, then $V_n$ is Lagrangian.
\end{enumerate}
\end{proof}

\subsubsection{$\mathcal{F}_{2n}^{\bi'}$ is a Schubert variety}

Let $W$ be a vector space over $\mathbb{C}$ of dimension $2(n+t)$ with a fixed basis $e_1,\ldots,e_{2(n+t)}$.

Let $\mathrm{SL}_{n}$ be the group of matrices over $\mathbb{C}$ of determinant $1$ of Lie algebra $\mathfrak{sl}_n$ consisting of traceless $n\times n$ matrices. Fix the Borel subgroup consisting of upper-triangular matrices in $\mathrm{SL}_n$ and the maximal torus consisting of diagonal matrices. The monoid of dominant integral weights will be denoted by $\Lambda^+_{\mathfrak{sl}_n}$ with fundamental weights $\varpi_1,\ldots,\varpi_{n-1}$ as generating set.

Let $[2n+2t]\setminus\{i_1+1,\ldots,i_{2t}+1\}=\{\ell_1,\ldots,\ell_{2n}\}$ with $\ell_1<\ldots<\ell_{2n}$. We define a morphism of monoids $\Theta:\Lambda^+_{\mathfrak{sl}_{2n}}\to \Lambda^+_{\mathfrak{sl}_{2(n+t)}}$ by sending $\varpi_k$ to $\varpi_{\ell_k}$. Let $\delta:=\varpi_1+\ldots+\varpi_{2n-1}\in\Lambda^+_{\mathfrak{sl}_{2n}}$ and $\delta':=\Theta(\delta)\in\Lambda^+_{\mathfrak{sl}_{2(n+t)}}$. Set $P^{\bi'}:=P_{\delta'}\subseteq \mathrm{SL}_{2(n+t)}$ be the stabilizer of the weight $\delta'$ in the group $\mathrm{SL}_{2(n+t)}$.

The partial flag variety $\mathrm{SL}_{2(n+t)}/P^{\bi'}$ admits the following Pl\"ucker embedding
$$\mathrm{SL}_{2(n+t)}/P^{\bi'}\hookrightarrow \prod_{k=1}^{2n}\mathrm{Gr}_{\ell_k}(W).$$

Let $t_1,\ldots,t_{2(n+t)-1}$ be the simple reflections in the Weyl group of $\mathrm{SL}_{2(n+t)}$ and let 
$v_k:=(t_k\cdots t_{i_k+k-1})\cdots (t_k\cdots t_{i_{k-1}+k})$ for $k=1,\ldots,2t+1$ where $i_0=0$ and $i_{2t+1}=2n-1$.
We define an element in the Weyl group of $\mathrm{SL}_{2(n+t)}$: $u_{\bi'}:=v_{2t+1}\cdots v_1$. The following result is proved in \cite[Proposition 6]{CFFFR}. Note that in the current setup, the number $n$ in \textit{loc.cit.} is $2n+2t-1$.

For $1\leq k\leq 2n$, we define $h_k:=\ell_k-k$. It follows immediately $0\leq h_k\leq 2t$ with $h_1=0$ and $h_{2n}=2t$. Note that if $h_{k+1}=h_k+1$, then $f_k'=\mathrm{pr}_{i_{h_{k+1}+1}}$.

\begin{lem}\label{Lem:ui}
\begin{enumerate}
\item If $\ell_j=\ell_{j-1}+1$, then $u_{\bi'}(\ell_j)=h_j+(2n+2t+1-j)$.
\item If $\ell_j=\ell_{j-1}+2$, then $u_{\bi'}(\ell_j-1)=h_j$ and $u_{\bi'}(\ell_j)=h_j+2n+2t$.
\end{enumerate}
\end{lem}

We first realize $\mathcal{F}_{2n}^{\bi'}$ as a Schubert variety. The proof of the following lemma is similar to that in \cite{CIL} with more involved combinatorics, we will emphasize on the part where necessary modifications are needed to be introduced.

\begin{lem}\label{Lem:TypeA}
There exists an isomorphism $\mathcal{F}_{2n}^{\bi'}\cong X_{u_{\bi'}}$, where $X_{u_{\bi'}}$ is the Schubert variety in $\mathrm{SL}_{2(n+t)}/P^{\bi'}$ associated to the Weyl group element $u_{\bi'}$.
\end{lem}

\begin{proof}
We start from defining a map 
$$\mu:\prod_{k=1}^{2n-1}\mathrm{Gr}_{k}(V)\to \prod_{k=1}^{2n-1}\mathrm{Gr}_{\ell_k}(W)$$
sending $\mathcal{F}_{2n}^{\bi'}$ to $X_{u_{\bi'}}$.

For $1\leq s\leq 2t$ set $U_{2n+s}=\langle e_1,\ldots,e_{2n+s}\rangle_{\mathbb{C}}$ and define a surjective linear map $\pi_s:U_{2n+s}\to V$ by:
\begin{enumerate}
\item[-] if $1\leq k\leq 2n$ with $k\neq i_1+1,\ldots,i_s+1$, then $\pi_s(e_k)=v_k$;
\item[-] $\pi_s(e_{i_1+1})=\ldots=\pi_s(e_{i_s+1})=0$;
\item[-] for $1\leq p\leq s$, $\pi_s(e_{2n+p})=v_{i_p+1}$.
\end{enumerate}
For convenience we set $U_{2n}:=\langle e_1,\ldots,e_{2n}\rangle_{\mathbb{C}}$ and $\pi_0:U_{2n}\to V$ sending $e_i$ to $v_i$ for $1\leq i\leq 2n$.

Define a map 
$$\mu_k:\mathrm{Gr}_k(V)\to\mathrm{Gr}_{\ell_k}(U_{2n+h_k})\to\mathrm{Gr}_{\ell_k}(W),$$
$$U\mapsto \pi_{h_k}^{-1}(U)\mapsto  \pi_{h_k}^{-1}(U).$$
Since $\langle e_{i_1+1},\ldots,e_{i_{h_k}+1}\rangle_{\mathbb{C}}\subseteq \pi_{h_k}^{-1}(U)$, the map $\mu_k$ is well-defined.

Gathering all $\mu_k$ gives the map $\mu:=(\mu_1,\ldots,\mu_{2n-1})$ above:
$$\prod_{k=1}^{2n-1}\mathrm{Gr}_{k}(V)\ni (V_1,\ldots,V_{2n-1})\mapsto (\pi_{h_1}^{-1}(V_1),\ldots,\pi_{h_{2n-1}}^{-1}(V_{2n-1}))\in\prod_{k=1}^{2n-1}\mathrm{Gr}_{\ell_k}(W).$$

Denote $W_k:=\pi_{h_k}^{-1}(V_k)$. To show that the image of $\mathcal{F}_{2n}^{\bi'}$ under $\mu$ is contained in the partial flag variety $\mathrm{SL}_{2(n+t)}/P^{\bi'}$, that is to say, $W_i\subseteq W_{i+1}$ for $1\leq i\leq 2n-2$, we need the following
\vskip 5pt
\noindent\textit{Claim}. For $1\leq k\leq 2n-2$, $\pi_{h_{k+1}}=f_k'\circ\pi_{h_k}$ on $U_{2n+h_k}$.

\begin{proof}
If $f_k'=\mathrm{id}$, then for any $1\leq p\leq 2t$, $k\neq i_p$ implies that $h_k=h_{k+1}$, it follows by definition that $\pi_{h_k}=\pi_{h_{k+1}}$.

If $f_k'=\mathrm{pr}_{i_p+1}$, then $h_{k+1}=h_k+1$. From the paragraph before Lemma \ref{Lem:ui}, $f_k'=\mathrm{pr}_{i_{h_{k+1}+1}}$. It remains to verify that $\pi_{h_{k+1}}=\mathrm{pr}_{i_{h_{k+1}+1}}\circ\pi_{h_k}$. Note that in this case, if we evaluate $\pi_{h_{k+1}}$ and $\pi_{h_k}$ on $e_1,\ldots,e_{2n+h_k}$, the only difference is the image of $e_{i_{h_{k+1}}+1}$, for the former the image is $0$, but for the latter the image is $v_{i_{h_{k+1}}+1}$. 
\end{proof}

Applying the claim, we have
$$W_i\subseteq \pi_{h_{i+1}}^{-1}\pi_{h_{i+1}}(W_i)=\pi_{h_{i+1}}^{-1} f_i'\pi_{h_i}(W_i)=\pi_{h_{i+1}}^{-1} f_i'(V_i)\subseteq \pi_{h_{i+1}}^{-1}(V_{i+1})=W_{i+1}.$$

It remains to show that the image of $\mathcal{F}_{2n}^{\bi'}$ under $\mu$ is exactly the Schubert variety $X_{u_{\bi'}}$.

According to Lemma \ref{Lem:ui}, the Schubert variety $X_{u_{\bi'}}$ is isomorphic to the tuple of subspaces $(W_1,\ldots,W_{2n-1})\in\prod_{j=1}^{2n-1} \mathrm{Gr}_{\ell_j}(W)$ satisfying $W_i\subseteq W_{i+1}$ for $1\leq i\leq 2n-2$ and 
\begin{equation}\label{Eq:Schubert}
\langle e_{i_{h_1}+1},\ldots,e_{i_{h_j}+1}\rangle_{\mathbb{C}}\subseteq W_j\subseteq \langle e_1,\ldots,e_{2n+h_j}\rangle_{\mathbb{C}}\ \ \text{for $1\leq j\leq 2n-1$}.
\end{equation}
This proves $\mu(\mathcal{F}_{2n}^{\bi'})\subseteq X_{u_{\bi'}}$.

For the other inclusion, take $(W_1,\ldots,W_{2n-1})\in X_{u_{\bi'}}$, it follows from \eqref{Eq:Schubert} that $\ker\pi_{h_j}\subseteq W_j\subseteq U_{2n+h_j}$. We define $V_j:=\pi_{h_j}(W_j)$, then from the above claim,
$$f_j'(V_j)=f_j'\pi_{h_j}(W_j)=\pi_{h_{j+1}}(W_j)\subseteq \pi_{h_{j+1}}(W_{j+1})=V_{j+1},$$
and hence $(V_1,\ldots,V_{2n-1})\in\mathcal{F}_{2n}^{\bi'}$.
\end{proof}

\subsubsection{The linear degenerate symplectic flag variety is a Schubert variety}

We first define an involution $\tau':\prod_{k=1}^{2n-1}\mathrm{Gr}_{\ell_k}(W)\to \prod_{k=1}^{2n-1}\mathrm{Gr}_{\ell_k}(W)$.

Note that until now we have not yet used the symplectic form on $W$. Our choice of such a form is uniquely determined by assigning for $1\leq k\leq 2t$,
$$\omega_W(e_{i_k+1},e_{2n+2t+1-k})=-1,$$
and for $1\leq s\leq n$ with $s\in [n]\setminus\{i_1+1,\ldots,i_{2t}+1\}$,
$$\omega_W(e_s,e_{2n+1-s})=1.$$
To remove the ambiguity, we will write $\bot_V$ and $\bot_W$ for the orthogonal complement in $V$ and $W$. 

The map $\tau'$ is defined by: 
$$\tau'(W_1,\ldots,W_{2n-1}):=(W_{2n-1}^{\bot_W},\ldots,W_1^{\bot_W}).$$
It follows from $\ell_k+\ell_{2n-k}=2n+2t$ that the map $\tau'$ is well-defined.

\begin{lem}\label{Lem:Commute}
The following diagram commutes:
\[
\xymatrix{
\prod_{k=1}^{2n-1}\mathrm{Gr}_{k}(V) \ar[d]^{\tau} \ar[r]^{\mu} & \prod_{k=1}^{2n-1}\mathrm{Gr}_{\ell_k}(W)\ar[d]^{\tau'}\\
\prod_{k=1}^{2n-1}\mathrm{Gr}_{k}(V) \ar[r]^{\mu} & \prod_{k=1}^{2n-1}\mathrm{Gr}_{\ell_k}(W).
}
\]
\end{lem}

Before giving the proof of the lemma, we first complete the proof of $(1)\Leftrightarrow (3)$. Lemma \ref{Lem:Commute}, together with Lemma \ref{Lem:TypeA}, implies that the invariants spaces are isomorphic:
$$(\mathcal{F}_{2n}^{\bi'})^\tau\cong (X_{u_{\bi'}})^{\tau'}.$$

It follows from \cite[Section 3]{LS2}, the definition of $u_{\bi'}$ and $w_\bi$ that $(X_{u_{\bi'}})^{\tau'}\cong X_{w_\bi}$. Applying Lemma \ref{Lem:SpInv} gives the isomorphism $\mathrm{Sp}\mathcal{F}_{2n}^\bi\cong X_{w_\bi}$.

\subsubsection{Proof of Lemma \ref{Lem:Commute}}
Fix $(V_1,\ldots,V_{2n-1})\in\prod_{k=1}^{2n-1}\mathrm{Gr}_{k}(V)$, it suffices to show that for any $1\leq k\leq 2n-1$,
$$\pi_{h_k}^{-1}(V_{2n-k}^{\bot_V})=\pi_{h_{2n-k}}^{-1}(V_{2n-k})^{\bot_W}.$$
It is clear that the vector spaces on both sides have the same dimension, it remains to verify the inclusion $\subseteq$.

For $1\leq j\leq 2t$, we define a linear map $p_j:V\to W$ by:
$$p_j(v_{i_\ell+1})=e_{2n+\ell},\ \ p_j(v_s)=e_s$$
for $1\leq \ell\leq j$ and $s\in [2n]\setminus\{i_1+1,\ldots,i_j+1\}$.

Take $u\in V_{2n-k}^{\bot_V}$ and $u\in V_{2n-k}$, then 
$$\pi_{h_k}^{-1}(v)=\underbrace{p_{h_k}(v)}_{\mathrm{(I)}}+\underbrace{\langle e_{i_1+1},\ldots,e_{i_k+1}\rangle_{\mathbb{C}}}_{\mathrm{(II)}},\ \ \pi_{h_{2n-k}}^{-1}(u)=\underbrace{p_{h_{2n-k}}(u)}_{\mathrm{(III)}}+\underbrace{\langle e_{i_1+1},\ldots,e_{i_{2n-k}+1}\rangle_{\mathbb{C}}}_{\mathrm{(IV)}},$$
It remains to show that $\omega_W(\pi_{h_k}^{-1}(v),\pi_{h_{2n-k}}^{-1}(u))=0$.

From the definition of $\omega_W$, the part (II) and (IV) are orthogonal. For the orthogonality of parts (I) and (IV), notice that the dual vector of $e_{i_{h_j}+1}$ for $1\leq j\leq 2n-k$ with respect to $\omega_W$ is $e_{2n+2t+1-h_j}$. It follows from $h_j+h_{2n-j}=2t$ and $1\leq j\leq 2n-k$ that $2t+1-h_j=h_{2n-j}+1$ with $k\leq 2n-j\leq 2n-1$, and hence $e_{2n+2t+1-h_j}$ is not in the image of $p_{h_k}$. Therefore (I) and (IV) are orthogonal. Similar argument can be applied to show that (II) and (III) are orthogonal.

It remains to show that $\omega_W(p_{h_k}(v),p_{h_{2n-k}}(u))=0$. Set $v=\sum_{i=1}^{2n}\lambda_iv_i$ and $u=\sum_{i=1}^{2n} \mu_i v_i$. It follows from $\omega_V(v,u)=0$ that 
\begin{equation}\label{Eq:Sum}
\sum_{i=1}^n\lambda_i\mu_{2n+1-i}-\sum_{i=1}^n \lambda_{2n+1-i}\mu_i=0.
\end{equation}
Set $C_j:=[2n]\setminus \{i_1+1,\ldots,i_j+1\}$ for $1\leq j\leq 2t$. After applying $p_{h_k}$ and $p_{h_{2n-k}}$ we obtain
$$p_{h_k}(v)=\underbrace{\sum_{i\in C_{h_k}}\lambda_i e_i}_{\mathrm{(A)}}+\underbrace{\sum_{j=1}^{h_k}\lambda_{i_j+1}e_{2n+j}}_{\mathrm{(B)}};\ \ p_{h_{2n-k}}(u)=\underbrace{\sum_{i\in C_{h_{2n-k}}}\mu_i e_i}_{\mathrm{(C)}}+\underbrace{\sum_{j=1}^{h_{2n-k}}\mu_{i_j+1}e_{2n+j}}_{\mathrm{(D)}}.$$

With respect to $\omega_W$, (B) and (D) are orthogonal. Taking the symplectic form of (B) with (C) gives $\sum_{j=1}^{h_k}\lambda_{i_j+1}\mu_{2n-i_j}$. Indeed, it suffices to notice that the dual vector of $e_{2n+j}$ is $e_{i_{2t+1-j}+1}$; if $1\leq j\leq h_k$, $e_{i_{2t+1-j}+1}\notin\{i_1+1,\ldots,i_{h_{2n-k}}+1\}$. Simiarly (A) with (D) gives $-\sum_{j=1}^{h_{2n-k}}\lambda_{2n-i_j}\mu_{i_j+1}$. The part (A) and (C) gives the remaining summands in \eqref{Eq:Sum} since for any $1\leq \ell\leq 2t$, $(i_\ell+1)+(i_{2t-\ell}+1)=2n+1$.

The proof of Lemma \ref{Lem:Commute} is then complete.

\subsection{\texorpdfstring{$(1)\Leftrightarrow (2)$}{1 equivalent 2 }}

Recall that $\be=(1,2,\ldots,2n-1)$. We will show that the following map 
$$\pi:\mathrm{Lag}_{\be}(M^\bi\oplus\nabla M^\bi)\to\mathrm{Sp}\mathcal{F}_{2n}^\bi,$$
$$(V_1,\ldots,V_{n-1},V_\omega,V_{(n-1)^*},\ldots,V_{1^*})\mapsto (V_1,\ldots,V_{n-1},V_\omega)$$
is an isomorphism.

Let $(V,\omega_V)$ be a symplectic vector space over $\mathbb{C}$ of dimension $2n$ and fix a basis $\{v_1,\ldots,v_{2n}\}$ of $V$ such that $\omega_V(v_i,v_{2n+1-j})=\delta_{i,j}$ for $1\leq i,j\leq n$. Using the symplectic form we can naturally identify $V$ and $V^*$ where the dual vector of $v_i$ is $v_i^*=v_{2n+1-i}$.

We define a $\stackrel{\longrightarrow}{A_{2n-1}}$-representation $M=(M_i,M_\alpha)_{i\in Q_0,\alpha\in Q_1}$ as follows:
\begin{enumerate}
\item[-] For $1\leq i\leq n-1$, $M_i=V$, $M_{i^*}=V^*$; $M_\omega=V$.
\item[-] For $1\leq i\leq n-1$ or $i=\omega$, fix the symplectic basis of $M_i$ and denote them by $\{v_{1,i},\ldots,v_{2n,i}\}$; fix the dual basis $\{v_{1,i}^*,\ldots,v_{2n,i}^*\}$ of $M_{i^*}$. 
\item[-] For $1\leq k\leq n-1$, the map $M_{\alpha_k}:M_k\to M_{k+1}$ where $M_n:=M_\omega$ and $v_{i,n}:=v_{i,\omega}$ is defined by: 
$$M_{\alpha_k}(v_{i,k})=\begin{cases} 
0, & \text{if $k\in\bi$, $i=k+1$;}\\
v_{i,k+1},& \text{otherwise}.\end{cases}$$
Note that $M_k$ and $M_{k+1}$ are the same space, $v_{i,k}$ and $v_{i,k+1}$ are the same basis element. Then $M_{\alpha_k}$ is the projection $\mathrm{pr}_{i_s+1}$ if $k=i_s$, and it is the identity map otherwise. The map $M_{\alpha_k^*}:M_{(k+1)^*}\to M_{k^*}$ where $M_{n^*}:=M_\omega$ is just the dual map of $M_{\alpha_k}$ with respect to the symplectic form. That is to say, if $k=i_s$ then $M_{\alpha_k^*}$ sends $v_{k+1,k}^*$ to zero and other $v_{\ell,k+1}^*$ for $\ell\neq k+1$ to $v_{\ell,k}^*$. It follows from $v_i^*=v_{2n+1-i}$ that $M_{\alpha_k^*}=\mathrm{pr}_{2n-i_s}$. If $k\notin \bi$, then $M_{\alpha_k^*}$ is the identity map.
\end{enumerate}

From the construction it follows that $M$ is isomorphic to $M^\bi\oplus \nabla M^\bi$ as $\stackrel{\longrightarrow}{A_{2n-1}}$-representations. Moreover, the linear maps $M_{\alpha_1},\ldots,M_{\alpha_{n-1}}$ are exactly $f_1,\ldots,f_{n-1}$. Taking a Lagrangian subrepresentation $(V_1,\ldots,V_{n-1},V_\omega,V_{(n-1)^*},\ldots,V_{1^*})$ of dimension vector $\be$ of $M$, it follows that for $1\leq i\leq n-1$, $\dim V_i=i$ and $f_i(V_i)\subseteq V_{i+1}$. Moreover $\dim V_\omega=n$, and $V_\omega$ is Lagrangian in $V$. This shows that the map $\pi$ is well-defined. Since the Lagrangian subrepresentation is uniquely determined by the part involving $V_1,\ldots,V_{n-1},V_\omega$ (because the remaining part consists of the quotient and the dual map), the map $\pi$ is an isomorphism.

We illustrate the proof in the example $n=3$ and $\bi=\{1\}$. The projective $\stackrel{\longrightarrow}{A_{2n-1}}$-representation is 
$$M^\bi=P_1^{\oplus 2}\oplus P_{1^*}\oplus P_2$$
and
$$M^\bi\oplus \nabla (M^\bi)=P_1^{\oplus 2}\oplus P_{1^*}\oplus P_2\oplus I_{2^*}\oplus I_1\oplus I_{1^*}^{\oplus 2}.$$

The coefficient quiver of $M$ is depicted below, where the bases are chosen as in the proof:
$$v_{6,1}\to v_{6,2} \to v_{6,\omega}\to v_{1,2}^*\to v_{1,1}^*$$
$$v_{5,1}\to v_{5,2} \to v_{5,\omega}\to v_{2,2}^*\hskip 16pt v_{2,1}^*$$
$$v_{4,1}\to v_{4,2} \to v_{4,\omega}\to v_{3,2}^*\to v_{3,1}^*$$
$$v_{3,1}\to v_{3,2} \to v_{3,\omega}\to v_{4,2}^*\to v_{4,1}^*$$
$$v_{2,1}\hskip 16pt v_{2,2} \to v_{2,\omega}\to v_{5,2}^*\to v_{5,1}^*$$
$$v_{1,1}\to v_{1,2} \to v_{1,\omega}\to v_{6,2}^*\to v_{6,1}^*.$$

\section{An algorithm for \texorpdfstring{$\ee$}{epsilon}-degenerations}\label{Sec_Algorithm}
We conclude the section by implementing the algorithm given in the proof of Lemma~\ref{Lem:CEquivalentR} in the case when $(Q,\sigma)$ is equioriented of type either $(\stackrel{\rightarrow}{A}_{\textrm{odd}},-1)$ or $(\stackrel{\rightarrow}{A}_{\textrm{even}},+1)$.

Let $M$ and $N$ be two $\ee$-representations of the same dimension vector. Let $(M^0,\langle-,-\rangle)$ be the underlying quadratic space of both $M$ and $N$. Assume that $r^M\geq r^N$.  We want to give a sequence $M=Z(0), Z(1),\cdots, Z(k+1)=N$ of $\ee$-representations and a sequence $(\lambda_i(t))_{i=0,\cdots, k}$ of one-parameter subgroups of the group of isometries of $(M^0,\langle-,-\rangle)$ such that $\lim_{t\rightarrow 0}\lambda_i(t)\cdot Z(i)=Z(i+1)$ (this sequence exists by Theorem~\ref{Thm:MainEquioriented}). 
We follow the algorithm of the proof of Lemma~\ref{Lem:CEquivalentR}:  At the $i$-th step, the algorithm has as input two $\ee$-representations $M(i)$ and $N(i)$  of the same dimension vector such that $r^{M(i)}\geq r^{N(i)}$ together with an indecomposable representation $L(i)$. They are constructed so that  there exists a one-parameter subgroup $\mu_i(t)$ which preserves the $\ee$-form such that $\lim_{t\rightarrow0}\mu_i(t)\cdot M(i)= M(i+1)\oplus L(i)\oplus \nabla L(i)$ and $N(i)=N(i+1)\oplus L(i)\oplus \nabla L(i)$. 
We then obtain the required sequence $\{Z_i\}$ by putting $Z(0)=M$ and $Z(i):=M(i)\oplus\bigoplus_{j<i} L(j)\oplus\nabla L(j)$ for $i>1$. 

The algorithm works as follows: if $r^M=r^N$ then $M\simeq N$; otherwise if $M_n$ is zero then both $M$ and $N$ are supported on a smaller quiver. We can hence assume that $M_n$ is not zero. We choose  the maximal index $i$ such that $r_{i,n}^N>r_{i-1,n}^N$; then $N=L(1)\oplus\nabla L(1)\oplus N(1)$ where $L(1)=U_{i,n}=P_i$. Let $\iota:L\rightarrow M$ be a generic isotropic embedding of $L$ into $M$ and we define $M(1)=\iota(L)^\perp/\iota(L)$. We then repeat this construction to the pair $M(1)$ and $N(1)$ and continue this way.
The algorithm stops when $r^{M(i)}=r^{N(i)}$. 

Tables \ref{Table:Ex1} and \ref{Table:Ex2}  provide two examples.

\begin{table}
\begin{tabular}{|c|c|c|c||c|c|}
\hline
\text{Step}&M(i)&N(i)&L(i+1)&Z(i)&\text{Coeff. quiver}\\&&&&&\text{of }Z(i)\\\hline&&&&&\\
(0)&
\xymatrix@R=0pt@C=0pt{
6&6&6&6&6\\
  &6&6&6&6\\
  &  &6&6&6\\
  &  &  &6&6\\
  &  &  &  &6}
&
\xymatrix@R=0pt@C=0pt{
6&5&4&3&2\\
  &6&5&4&3\\
  &  &6&5&4\\
  &  &  &6&5\\
  &  &  &  &*+[F]{6}}
&
$P_5=U_{5,5}$
&
\xymatrix@R=0pt@C=0pt{
6&6&6&6&6\\
  &6&6&6&6\\
  &  &6&6&6\\
  &  &  &6&6\\
  &  &  &  &6}
&
\xymatrix@R=3pt@C=10pt{
*{\cdot}\ar@{-}[r]&*{\cdot}\ar@{-}[r]&*{\cdot}\ar@{-}[r]&*{\cdot}\ar@{-}[r]&*{\cdot}\\
*{\cdot}\ar@{-}[r]&*{\cdot}\ar@{-}[r]&*{\cdot}\ar@{-}[r]&*{\cdot}\ar@{-}[r]&*{\cdot}\\
*{\cdot}\ar@{-}[r]&*{\cdot}\ar@{-}[r]&*{\cdot}\ar@{-}[r]&*{\cdot}\ar@{-}[r]&*{\cdot}\\
*{\cdot}\ar@{-}[r]&*{\cdot}\ar@{-}[r]&*{\cdot}\ar@{-}[r]&*{\cdot}\ar@{-}[r]&*{\cdot}\\
*{\cdot}\ar@{-}[r]&*{\cdot}\ar@{-}[r]&*{\cdot}\ar@{-}[r]&*{\cdot}\ar@{-}[r]&*{\cdot}\\
*{\cdot}\ar@{-}[r]&*{\cdot}\ar@{-}[r]&*{\cdot}\ar@{-}[r]&*{\cdot}\ar@{-}[r]&*{\cdot}}
\\\hline&&&&&\\
(1)&
\xymatrix@R=0pt@C=0pt{
5&5&5&5&4\\
  &6&6&6&5\\
  &  &6&6&5\\
  &  &  &6&5\\
  &  &  &  &5}
&
\xymatrix@R=0pt@C=0pt{
5&5&4&3&2\\
  &6&5&4&3\\
  &  &6&5&4\\
  &  &  &6&*+[F]{5}\\
  &  &  &  &5}
&
$P_4=U_{4,5}$
&
\xymatrix@R=0pt@C=0pt{
6&5&5&5&4\\
  &6&6&6&5\\
  &  &6&6&5\\
  &  &  &6&5\\
  &  &  &  &6}
  &
\xymatrix@R=3pt@C=10pt{
*{\cdot}\ar@{-}[r]&*{\cdot}\ar@{-}[r]&*{\cdot}\ar@{-}[r]&*{\cdot}\ar@{-}[r]&*{\cdot}\\
*{\cdot}\ar@{-}[r]&*{\cdot}\ar@{-}[r]&*{\cdot}\ar@{-}[r]&*{\cdot}&*{\cdot}\\
*{\cdot}\ar@{-}[r]&*{\cdot}\ar@{-}[r]&*{\cdot}\ar@{-}[r]&*{\cdot}\ar@{-}[r]&*{\cdot}\\
*{\cdot}\ar@{-}[r]&*{\cdot}\ar@{-}[r]&*{\cdot}\ar@{-}[r]&*{\cdot}\ar@{-}[r]&*{\cdot}\\
*{\cdot}&*{\cdot}\ar@{-}[r]&*{\cdot}\ar@{-}[r]&*{\cdot}\ar@{-}[r]&*{\cdot}\\
*{\cdot}\ar@{-}[r]&*{\cdot}\ar@{-}[r]&*{\cdot}\ar@{-}[r]&*{\cdot}\ar@{-}[r]&*{\cdot}}
\\\hline&&&&&\\
(2)&
\xymatrix@R=0pt@C=0pt{
4&4&4&4&4\\
  &5&5&4&4\\
  &  &6&5&4\\
  &  &  &5&4\\
  &  &  &  &4}
&
\xymatrix@R=0pt@C=0pt{
4&4&4&3&2\\
  &5&5&4&3\\
  &  &6&5&*+[F]{4}\\
  &  &  &5&4\\
  &  &  &  &4}
&
$P_3=U_{3,5}$
&
\xymatrix@R=0pt@C=0pt{
6&5&4&4&4\\
  &6&5&4&4\\
  &  &6&5&4\\
  &  &  &6&5\\
  &  &  &  &6}
  &
\xymatrix@R=3pt@C=10pt{
*{\cdot}\ar@{-}[r]&*{\cdot}\ar@{-}[r]&*{\cdot}\ar@{-}[r]&*{\cdot}\ar@{-}[r]&*{\cdot}\\
*{\cdot}\ar@{-}[r]&*{\cdot}&*{\cdot}\ar@{-}[r]&*{\cdot}&*{\cdot}\\
*{\cdot}\ar@{-}[r]&*{\cdot}\ar@{-}[r]&*{\cdot}\ar@{-}[r]&*{\cdot}\ar@{-}[r]&*{\cdot}\\
*{\cdot}\ar@{-}[r]&*{\cdot}\ar@{-}[r]&*{\cdot}\ar@{-}[r]&*{\cdot}\ar@{-}[r]&*{\cdot}\\
*{\cdot}&*{\cdot}\ar@{-}[r]&*{\cdot}&*{\cdot}\ar@{-}[r]&*{\cdot}\\
*{\cdot}\ar@{-}[r]&*{\cdot}\ar@{-}[r]&*{\cdot}\ar@{-}[r]&*{\cdot}\ar@{-}[r]&*{\cdot}}
\\\hline&&&&&\\
(3)&
\xymatrix@R=0pt@C=0pt{
3&3&3&3&2\\
  &4&4&4&3\\
  &  &4&4&3\\
  &  &  &4&3\\
  &  &  &  &3}
&
\xymatrix@R=0pt@C=0pt{
3&3&3&3&2\\
  &4&4&4&3\\
  &  &4&4&3\\
  &  &  &4&3\\
  &  &  &  &3}
&
&
\xymatrix@R=0pt@C=0pt{
6&5&4&3&2\\
  &6&5&4&3\\
  &  &6&5&4\\
  &  &  &6&5\\
  &  &  &  &6}
  &
\xymatrix@R=3pt@C=10pt{
*{\cdot}\ar@{-}[r]&*{\cdot}\ar@{-}[r]&*{\cdot}\ar@{-}[r]&*{\cdot}\ar@{-}[r]&*{\cdot}\\
*{\cdot}\ar@{-}[r]&*{\cdot}\ar@{-}[r]&*{\cdot}\ar@{-}[r]&*{\cdot}&*{\cdot}\\
*{\cdot}\ar@{-}[r]&*{\cdot}\ar@{-}[r]&*{\cdot}&*{\cdot}\ar@{-}[r]&*{\cdot}\\
*{\cdot}\ar@{-}[r]&*{\cdot}&*{\cdot}\ar@{-}[r]&*{\cdot}\ar@{-}[r]&*{\cdot}\\
*{\cdot}&*{\cdot}\ar@{-}[r]&*{\cdot}\ar@{-}[r]&*{\cdot}\ar@{-}[r]&*{\cdot}\\
*{\cdot}\ar@{-}[r]&*{\cdot}\ar@{-}[r]&*{\cdot}\ar@{-}[r]&*{\cdot}\ar@{-}[r]&*{\cdot}}
\\\hline
\end{tabular}
\caption{An example of the algorithm in type ($\stackrel{\rightarrow}{A_5}, -1$)}
\label{Table:Ex1}
\end{table}

\begin{table}
\begin{tabular}{|c|c|c|c|}
\hline
\textrm{Step}&M(i)&N(i)&L\\\hline&&&\\
(0)&
\xymatrix@R=0pt@C=0pt{
3&3&2&2&2\\
  &3&2&2&2\\
  &  &4&2&2\\
  &  &  &3&3\\
  &  &  &  &3}
&
\xymatrix@R=0pt@C=0pt{
3&1&1&0&0\\
  &3&1&0&0\\
  &  &4&1&*+[F]{1}\\
  &  &  &3&1\\
  &  &  &  &3}
&
$P_3=U_{3,5}$
\\\hline&&&\\
(1)&\xymatrix@R=0pt@C=0pt{
2&2&1&0&0\\
  &2&1&0&0\\
  &  &2&1&1\\
  &  &  &2&2\\
  &  &  &  &2}
&
\xymatrix@R=0pt@C=0pt{
2&0&0&0&0\\
  &2&0&0&0\\
  &  &2&0&0\\
  &  &  &2&0\\
  &  &  &  &*+[F]{2}}
&
$P_5=U_{5,5}$
\\\hline&&&\\
(2)&\xymatrix@R=0pt@C=0pt{
1&1&1&0&0\\
  &2&1&0&0\\
  &  &2&1&1\\
  &  &  &2&1\\
  &  &  &  &1}
&
\xymatrix@R=0pt@C=0pt{
1&0&0&0&0\\
  &2&0&0&0\\
  &  &2&0&0\\
  &  &  &2&0\\
  &  &  &  &*+[F]{1}}
&
$P_5=U_{5,5}$
\\\hline&&&\\
(3)&\xymatrix@R=0pt@C=0pt{
0&0&0&0&0\\
  &2&1&0&0\\
  &  &2&1&0\\
  &  &  &2&0\\
  &  &  &  &0}
&
\xymatrix@R=0pt@C=0pt{
0&0&0&0&0\\
  &2&0&0&0\\
  &  &2&0&0\\
  &  &  &*+[F]{2}&0\\
  &  &  &  &0}
&
$S_4=U_{4,4}$
\\\hline&&&\\
(4)&\xymatrix@R=0pt@C=0pt{
0&0&0&0&0\\
  &1&1&0&0\\
  &  &2&1&0\\
  &  &  &1&0\\
  &  &  &  &0}
&
\xymatrix@R=0pt@C=0pt{
0&0&0&0&0\\
  &1&0&0&0\\
  &  &2&0&0\\
  &  &  &*+[F]{1}&0\\
  &  &  &  &0}
&
$S_4=U_{4,4}$
\\\hline&&&\\
(5)&\xymatrix@R=0pt@C=0pt{
0&0&0&0&0\\
  &0&0&0&0\\
  &  &2&0&0\\
  &  &  &0&0\\
  &  &  &  &0}
&
\xymatrix@R=0pt@C=0pt{
0&0&0&0&0\\
  &0&0&0&0\\
  &  &2&0&0\\
  &  &  &0&0\\
  &  &  &  &0}
&
\\\hline
\end{tabular}
\caption{An example of the algorithm in type ($\stackrel{\rightarrow}{A_5}, -1$)}
\label{Table:Ex2}
\end{table}
\printbibliography

@article {AbeasisDelFraEquioriented,
    AUTHOR = {Abeasis, S. and Del Fra, A.},
     TITLE = {Degenerations for the representations of an equioriented
              quiver of type {$A\sb{m}$}},
   JOURNAL = {Boll. Un. Mat. Ital. Suppl.},
  FJOURNAL = {Unione Matematica Italiana. Bollettino. Supplemento},
      YEAR = {1980},
    NUMBER = {2},
     PAGES = {157--171},
   MRCLASS = {16A64 (16A58)},
  MRNUMBER = {675498},
MRREVIEWER = {Sheila\ Brenner},
}

@misc {BCI,
author = { Boos, Magdalena and Cerulli-Irelli, Giovanni},
title = {On degenerations and extensions of symplectic and orthogonal quiver representations },
note = {arXiv:2106.08666}
}

@article {BFK,
    AUTHOR = {Balla, George and Fourier, Ghislain and Kambaso, Kunda},
     TITLE = {P{BW} filtration and monomial bases for {D}emazure modules in
              types {A} and {C}},
   JOURNAL = {Beitr. Algebra Geom.},
  FJOURNAL = {Beitr\"{a}ge zur Algebra und Geometrie. Contributions to
              Algebra and Geometry},
    VOLUME = {64},
      YEAR = {2023},
    NUMBER = {4},
     PAGES = {887--907},
      ISSN = {0138-4821,2191-0383},
   MRCLASS = {17B10 (05E10 14M15)},
  MRNUMBER = {4652887},
       DOI = {10.1007/s13366-022-00660-0},
       URL = {https://doi.org/10.1007/s13366-022-00660-0},
}

@article {Bongartz,
    AUTHOR = {Bongartz, Klaus},
     TITLE = {On degenerations and extensions of finite-dimensional modules},
   JOURNAL = {Adv. Math.},
  FJOURNAL = {Advances in Mathematics},
    VOLUME = {121},
      YEAR = {1996},
    NUMBER = {2},
     PAGES = {245--287},
      ISSN = {0001-8708,1090-2082},
   MRCLASS = {16G10 (14L30 16G60 16G70)},
  MRNUMBER = {1402728},
MRREVIEWER = {Christof\ Gei\ss },
       DOI = {10.1006/aima.1996.0053},
       URL = {https://doi.org/10.1006/aima.1996.0053},
}

@article {CIEFR21,
    AUTHOR = {Cerulli Irelli, Giovanni and Esposito, Francesco and Franzen,
              Hans and Reineke, Markus},
     TITLE = {Cell decompositions and algebraicity of cohomology for quiver
              {G}rassmannians},
   JOURNAL = {Adv. Math.},
  FJOURNAL = {Advances in Mathematics},
    VOLUME = {379},
      YEAR = {2021},
     PAGES = {Paper No. 107544, 47},
      ISSN = {0001-8708,1090-2082},
   MRCLASS = {14M15 (13F60 14C15 16G20)},
  MRNUMBER = {4198640},
MRREVIEWER = {Xiaobo\ Zhuang},
       DOI = {10.1016/j.aim.2020.107544},
       URL = {https://doi.org/10.1016/j.aim.2020.107544},
}

@article {CFFFR,
    AUTHOR = {Cerulli Irelli, Giovanni and Fang, Xin and Feigin, Evgeny and Fourier, Ghislain
              and Reineke, Markus},
     TITLE = {Linear degenerations of flag varieties},
   JOURNAL = {Math. Z.},
  FJOURNAL = {Mathematische Zeitschrift},
    VOLUME = {287},
      YEAR = {2017},
    NUMBER = {1-2},
     PAGES = {615--654},
      ISSN = {0025-5874,1432-1823},
   MRCLASS = {14M15 (14D06)},
  MRNUMBER = {3694690},
MRREVIEWER = {Ryan\ David\ Kinser},
       DOI = {10.1007/s00209-016-1839-y},
       URL = {https://doi.org/10.1007/s00209-016-1839-y},
}

@article {CFFFR20,
    AUTHOR = {Cerulli Irelli, Giovanni and Fang, Xin and Feigin, Evgeny and
              Fourier, Ghislain and Reineke, Markus},
     TITLE = {Linear degenerations of flag varieties: partial flags,
              defining equations, and group actions},
   JOURNAL = {Math. Z.},
  FJOURNAL = {Mathematische Zeitschrift},
    VOLUME = {296},
      YEAR = {2020},
    NUMBER = {1-2},
     PAGES = {453--477},
      ISSN = {0025-5874,1432-1823},
   MRCLASS = {14M15 (14D06)},
  MRNUMBER = {4140749},
MRREVIEWER = {Sudarshan\ Rajendra\ Gurjar},
       DOI = {10.1007/s00209-019-02451-1},
       URL = {https://doi.org/10.1007/s00209-019-02451-1},
}

@article {CIFR12,
    AUTHOR = {Cerulli Irelli, Giovanni and Feigin, Evgeny and Reineke,
              Markus},
     TITLE = {Quiver {G}rassmannians and degenerate flag varieties},
   JOURNAL = {Algebra Number Theory},
  FJOURNAL = {Algebra \& Number Theory},
    VOLUME = {6},
      YEAR = {2012},
    NUMBER = {1},
     PAGES = {165--194},
      ISSN = {1937-0652,1944-7833},
   MRCLASS = {14M15 (16G20)},
  MRNUMBER = {2950163},
MRREVIEWER = {Nicholas\ J.\ Proudfoot},
       DOI = {10.2140/ant.2012.6.165},
       URL = {https://doi.org/10.2140/ant.2012.6.165},
}

@article {CIEM23,
    AUTHOR = {Cerulli Irelli, Giovanni and Esposito, Francesco and Marietti,
              Mario},
     TITLE = {Motzkin combinatorics in linear degenerations of the flag
              variety},
   JOURNAL = {Int. Math. Res. Not. IMRN},
  FJOURNAL = {International Mathematics Research Notices. IMRN},
      YEAR = {2023},
    NUMBER = {22},
     PAGES = {19184--19204},
      ISSN = {1073-7928,1687-0247},
   MRCLASS = {14M15},
  MRNUMBER = {4669799},
       DOI = {10.1093/imrn/rnad063},
       URL = {https://doi.org/10.1093/imrn/rnad063},
}

@article {CIFR17,
    AUTHOR = {Cerulli Irelli, Giovanni and Feigin, Evgeny and Reineke,
              Markus},
     TITLE = {Schubert quiver {G}rassmannians},
   JOURNAL = {Algebr. Represent. Theory},
  FJOURNAL = {Algebras and Representation Theory},
    VOLUME = {20},
      YEAR = {2017},
    NUMBER = {1},
     PAGES = {147--161},
      ISSN = {1386-923X,1572-9079},
   MRCLASS = {14M15 (16G20)},
  MRNUMBER = {3606484},
MRREVIEWER = {Kavita\ Sutar},
       DOI = {10.1007/s10468-016-9634-3},
       URL = {https://doi.org/10.1007/s10468-016-9634-3},
}

@article {CIL,
    AUTHOR = {Cerulli Irelli, Giovanni and Lanini, Martina},
     TITLE = {Degenerate flag varieties of type {A} and {C} are {S}chubert
              varieties},
   JOURNAL = {Int. Math. Res. Not. IMRN},
      YEAR = {2015},
    NUMBER = {15},
     PAGES = {6353--6374},
       DOI = {10.1093/imrn/rnu128},
}

@article {DW,
    AUTHOR = {Derksen, Harm and Weyman, Jerzy},
     TITLE = {Generalized quivers associated to reductive groups},
   JOURNAL = {Colloq. Math.},
  FJOURNAL = {Colloquium Mathematicum},
    VOLUME = {94},
      YEAR = {2002},
    NUMBER = {2},
     PAGES = {151--173},
      ISSN = {0010-1354,1730-6302},
   MRCLASS = {16G20 (14L35)},
  MRNUMBER = {1967372},
MRREVIEWER = {Iain\ G.\ Gordon},
       DOI = {10.4064/cm94-2-1},
       URL = {https://doi.org/10.4064/cm94-2-1},
}

@misc {EFFS,
    AUTHOR = {Enugandla, Shreepranav Varma and Fang, Xin and Fourier, Ghislain and Steinert, Christian},
     TITLE = {Dynkin abelianizations of flag varieties},
   note = {arXiv:2404.05277}
}

@misc {CIEFF22,
    AUTHOR = {Cerulli Irelli, Giovanni and Esposito, Franceso and Fang, Xin and Fourier, Ghislain},
     TITLE = {Specialization map for quiver Grassmannians},
   note = {arXiv:2206.10281}
}

@article {Rei13,
    AUTHOR = {Reineke, Markus},
     TITLE = {Every projective variety is a quiver {G}rassmannian},
   JOURNAL = {Algebr. Represent. Theory},
  FJOURNAL = {Algebras and Representation Theory},
    VOLUME = {16},
      YEAR = {2013},
    NUMBER = {5},
     PAGES = {1313--1314},
      ISSN = {1386-923X,1572-9079},
   MRCLASS = {14A10 (16G20)},
  MRNUMBER = {3102955},
MRREVIEWER = {Ada\ Boralevi},
       DOI = {10.1007/s10468-012-9357-z},
       URL = {https://doi.org/10.1007/s10468-012-9357-z},
}

@article {Fei12,
    AUTHOR = {Feigin, Evgeny},
     TITLE = {{$\mathbb{G}_a^M$} degeneration of flag varieties},
   JOURNAL = {Selecta Math. (N.S.)},
  FJOURNAL = {Selecta Mathematica. New Series},
    VOLUME = {18},
      YEAR = {2012},
    NUMBER = {3},
     PAGES = {513--537},
      ISSN = {1022-1824,1420-9020},
   MRCLASS = {14L35 (17B45)},
  MRNUMBER = {2960025},
MRREVIEWER = {Oksana\ S.\ Yakimova},
       DOI = {10.1007/s00029-011-0084-9},
       URL = {https://doi.org/10.1007/s00029-011-0084-9},
}

@article {FeFiL,
    AUTHOR = {Feigin, Evgeny and Finkelberg, Michael and Littelmann, Peter},
     TITLE = {Symplectic degenerate flag varieties},
   JOURNAL = {Canad. J. Math.},
  FJOURNAL = {Canadian Journal of Mathematics. Journal Canadien de
              Math\'{e}matiques},
    VOLUME = {66},
      YEAR = {2014},
    NUMBER = {6},
     PAGES = {1250--1286},
      ISSN = {0008-414X,1496-4279},
   MRCLASS = {14M15 (17B10 20G05)},
  MRNUMBER = {3270783},
MRREVIEWER = {Jacopo\ Gandini},
       DOI = {10.4153/CJM-2013-038-6},
       URL = {https://doi.org/10.4153/CJM-2013-038-6},
}

@article {FR21,
    AUTHOR = {Fang, Xin and Reineke, Markus},
     TITLE = {Supports for linear degenerations of flag varieties},
   JOURNAL = {Doc. Math.},
  FJOURNAL = {Documenta Mathematica},
    VOLUME = {26},
      YEAR = {2021},
     PAGES = {1981--2003},
      ISSN = {1431-0635,1431-0643},
   MRCLASS = {14M15 (14D06 14F06 17B37)},
  MRNUMBER = {4378719},
MRREVIEWER = {Qifeng\ Li},
}

@misc {Mak19,
    AUTHOR = {Maksimau, Ruslan},
     TITLE = {Flag versions of quiver Grassmannians for Dynkin quivers have no odd cohomology over $\mathbb{Z}$},
   note = {arXiv:1909.04907}
}

@article {LS19,
    AUTHOR = {Lanini, Martina and Strickland, Elisabetta},
     TITLE = {Cohomology of the flag variety under {PBW} degenerations},
   JOURNAL = {Transform. Groups},
  FJOURNAL = {Transformation Groups},
    VOLUME = {24},
      YEAR = {2019},
    NUMBER = {3},
     PAGES = {835--844},
      ISSN = {1083-4362,1531-586X},
   MRCLASS = {14M15 (14D06)},
  MRNUMBER = {3989693},
MRREVIEWER = {Roberto\ Mu\~{n}oz},
       DOI = {10.1007/s00031-018-9484-7},
       URL = {https://doi.org/10.1007/s00031-018-9484-7},
}

@article {LS2,
    AUTHOR = {Lakshmibai, V. and Seshadri, C. S.},
     TITLE = {Geometry of {$G/P$}. {II}. {T}he work of de {C}oncini and
              {P}rocesi and the basic conjectures},
   JOURNAL = {Proc. Indian Acad. Sci. Sect. A},
  FJOURNAL = {Proceedings of the Indian Academy of Sciences. Section A},
    VOLUME = {87},
      YEAR = {1978},
    NUMBER = {2},
     PAGES = {1--54},
      ISSN = {0370-0089},
   MRCLASS = {14M15 (14L30 14M17)},
  MRNUMBER = {490244},
MRREVIEWER = {H.\ H.\ Andersen},
}

@article {R,
    AUTHOR = {Richardson, R. W.},
     TITLE = {On orbits of algebraic groups and {L}ie groups},
   JOURNAL = {Bull. Austral. Math. Soc.},
  FJOURNAL = {Bulletin of the Australian Mathematical Society},
    VOLUME = {25},
      YEAR = {1982},
    NUMBER = {1},
     PAGES = {1--28},
      ISSN = {0004-9727},
   MRCLASS = {14L30 (22E15)},
  MRNUMBER = {651417},
MRREVIEWER = {Klaus\ Pommerening},
       DOI = {10.1017/S0004972700005013},
       URL = {https://doi.org/10.1017/S0004972700005013},
}

@article {V,
    AUTHOR = {Vinberg, \`E. B.},
     TITLE = {The {W}eyl group of a graded {L}ie algebra},
   JOURNAL = {Izv. Akad. Nauk SSSR Ser. Mat.},
  FJOURNAL = {Izvestiya Akademii Nauk SSSR. Seriya Matematicheskaya},
    VOLUME = {40},
      YEAR = {1976},
    NUMBER = {3},
     PAGES = {488--526, 709},
      ISSN = {0373-2436},
   MRCLASS = {22E60},
  MRNUMBER = {430168},
MRREVIEWER = {G.\ Lusztig},
}
\end{document}